\documentclass{amsart}%
\usepackage{amssymb}
\usepackage{amsmath}
\usepackage{amsfonts}
\usepackage[all,cmtip]{xy}
\usepackage[margin=1.28in]{geometry}
\usepackage{hyperref}
\usepackage{graphicx}%
\newtheorem{theorem}{Theorem}
\newtheorem*{theoremannounce}{Theorem}
\numberwithin{theorem}{section}
\theoremstyle{plain}
\newtheorem*{acknowledgement}{Acknowledgement}

\newtheorem{definition}[theorem]{Definition}

\newtheorem{lemma}[theorem]{Lemma}

\theoremstyle{remark}
\newtheorem{conjecture}{Conjecture}

\newtheorem*{convention}{Convention}
\newtheorem{remark}[theorem]{Remark}

\numberwithin{equation}{section}

\begin{document}
\title[The relative $K$-group in the ETNC, Part II]{On the relative $K$-group in the ETNC, Part II}
\author{Oliver Braunling}
\address{Mathematical Institute, University of Bonn, Endenicher Allee 60, 53115 Bonn, Germany}
\email{oliver.braeunling@math.uni-freiburg.de, obraeunl@uni-bonn.de}
\thanks{The author was supported by DFG GK1821 \textquotedblleft Cohomological Methods
in Geometry\textquotedblright\ and a Junior Fellowship at the Freiburg
Institute for Advanced Studies (FRIAS)}

\begin{abstract}
In a previous paper we showed, under some assumptions, that the relative
$K$-group in the Burns-Flach formulation of the\ equivariant Tamagawa number
conjecture (ETNC) is canonically isomorphic to a $K$-group of locally compact
equivariant modules. This viewpoint, as well as the usual one, come with
generator-relator presentations (due to Bass--Swan and Nenashev) and in this
paper we provide an explicit map.

\end{abstract}
\maketitle

Let $A$ be a finite-dimensional semisimple $\mathbb{Q}$-algebra and
$\mathfrak{A}\subset A$ an order. We write $A_{\mathbb{R}}:=A\otimes
_{\mathbb{Q}}\mathbb{R}$. There is the long exact sequence%
\begin{equation}
\cdots\longrightarrow K_{n}(\mathfrak{A})\longrightarrow K_{n}(A_{\mathbb{R}%
})\longrightarrow K_{n-1}(\mathfrak{A},\mathbb{R})\longrightarrow
K_{n-1}(\mathfrak{A})\longrightarrow\cdots\text{.}\label{ler2}%
\end{equation}
In the previous paper \cite{etnclca} we have shown, assuming $\mathfrak{A}$ to
be regular, that there is a canonical isomorphism%
\begin{equation}
K_{n}(\mathfrak{A},\mathbb{R})\cong K_{n+1}(\mathsf{LCA}_{\mathfrak{A}%
})\text{,}\label{ler1}%
\end{equation}
where $\mathsf{LCA}_{\mathfrak{A}}$ is the exact category of locally compact
topological right $\mathfrak{A}$-modules. We are mostly interested in the case
$n=0$. In the formulation of the equivariant Tamagawa number conjecture (ETNC)
of Burns and Flach \cite{MR1884523}, the equivariant Tamagawa numbers live in
the relative $K$-group $K_{0}(\mathfrak{A},\mathbb{R})$. This group has an
explicit presentation due to Bass--Swan, based on generators%
\[
\lbrack P,\varphi,Q]\qquad\qquad\text{with}\qquad\qquad\varphi:P_{\mathbb{R}%
}\overset{\sim}{\longrightarrow}Q_{\mathbb{R}}\text{,}%
\]
where $P,Q$ are finitely generated projective right $\mathfrak{A}$-modules;
modulo some relations. On the other hand, the group $K_{1}(\mathsf{LCA}%
_{\mathfrak{A}})$ has an explicit presentation due to Nenashev, based on
double exact sequences%
\[%
\xymatrix{
A \ar@<1ex>@{^{(}->}[r]^{p} \ar@<-1ex>@{^{(}.>}[r]_{q} & B \ar@<1ex>@{->>}%
[r]^{r} \ar@<-1ex>@{.>>}[r]_{s} & C
}%
\text{,}%
\]
where $A,B,C$ are locally compact right $\mathfrak{A}$-modules; modulo some
other relations. Unfortunately, the isomorphism between these groups, and
between these two concrete presentations, produced by \cite{etnclca} had
remained inexplicit. We fix this issue now.\medskip

Given $[P,\varphi,Q]$, we send it to a double exact sequence called
\textquotedblleft$\left\langle \left\langle P,\varphi,Q\right\rangle
\right\rangle $\textquotedblright, having the shape%
\[
\left[
\xymatrix{
{\prod T_{P}\oplus P\oplus Q\oplus\prod T_{Q}} \ar@<1ex>@{^{(}->}%
[r] \ar@<-1ex>@{^{(}.>}[r] & {\prod T_{P}\oplus P\oplus P_{\mathbb{R}}\oplus
Q\oplus\prod T_{Q}} \ar@<1ex>@{->>}[r] \ar@<-1ex>@{.>>}[r] & {T_{P}%
\oplus\bigoplus P\oplus\bigoplus Q\oplus T_{Q}}
}%
\right]  \text{,}%
\]
where $P$ carries the discrete topology, $P_{\mathbb{R}}:=P\otimes\mathbb{R}$
is topologically a real vector space, and $T_{P}:=P_{\mathbb{R}}/P$ carries
the torus quotient topology (correspondingly for $Q$). The products and direct
sums are indexed over $\mathbb{Z}_{\geq0}$. The four morphisms are complicated
to define: They arise as the sum of various maps, which can be depicted as%
\[%
{\includegraphics[
height=1.855in,
width=2.4154in
]%
{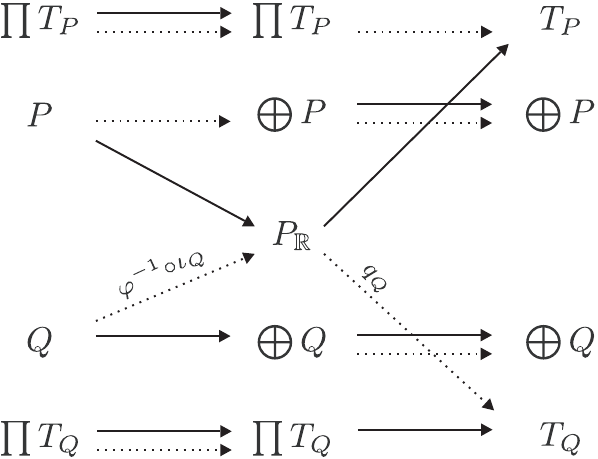}%
}
\]
where we read the objects in the left (resp. middle, right)\ column as the
direct summands appearing in the left (resp. middle, right) term of the double
exact sequence. The solid versus dotted lines indicate which side of the
double exact sequence they belong to. Each of the arrows without a label only
depends on $P$ and $Q$, but not on $\varphi$. A detailed construction of this
double exact sequence is too long for the introduction, and we refer to
\S \ref{sect_ExplicitMap} for details.\ We only gave these brief indications
to give an impression what kind of object we are dealing with. It is complicated.

But it is exactly what we need for an explicit formula:

\begin{theoremannounce}
Let $A$ be a finite-dimensional semisimple $\mathbb{Q}$-algebra and
$\mathfrak{A}\subset A$ an order. Then the map%
\begin{equation}
K_{0}(\mathfrak{A},\mathbb{R})\longrightarrow K_{1}(\mathsf{LCA}%
_{\mathfrak{A}})\label{ler2a}%
\end{equation}
sending $[P,\varphi,Q]$ to the double exact sequence $\left\langle
\left\langle P,\varphi,Q\right\rangle \right\rangle $ is a well-defined
morphism from the Bass--Swan to the Nenashev presentation. If $\mathfrak{A}$
is regular, then this map is an isomorphism.
\end{theoremannounce}

The mere existence of the map (not needing $\mathfrak{A}$ regular) will be
Theorem \ref{thm_VarThetaWellDefined} and is independent of the previous paper
\cite{etnclca}. The proof of isomorphy will be Theorem \ref{thm_IdentifySeq}.
The latter is a lot harder and will require us to work with some explicit
simplicial model of $K$-theory, based on the work of Gillet and
Grayson.\medskip

Why is this comparison map so complicated?\medskip

Indeed, the apparent complexity of $\left\langle \left\langle P,\varphi
,Q\right\rangle \right\rangle $ is very misleading. It feels appropriate to
compare this situation to Leibniz's formula for the determinant%
\[
\det M=\sum_{\pi\in\mathfrak{S}_{n}}\operatorname*{sgn}(\pi)M_{1,\pi(1)}\cdot
M_{2,\pi(2)}\cdots M_{n,\pi(n)}\text{.}%
\]
Hardly anyone would use this formula for working with determinants in a
concrete computation, or for any practical purpose, but it is a closed-form
formula which is valid for \textit{all} matrices $M$. The same is true for
$\left\langle \left\langle P,\varphi,Q\right\rangle \right\rangle $. In
general, it is an overcomplicated representative for the underlying $K_{1}%
$-class, but it is a \textit{uniform} closed-form expression valid for all
Bass--Swan representatives $[P,\varphi,Q]$.\medskip

Since the ETNC is a conjecture in arithmetic geometry, readers might not be
particularly enthusiastic about simplicial sets. We have gone to some lengths
in order to write down the proof of the above theorem as a direct verification
using only the Bass--Swan and Nenashev presentations alone, mostly. Yet, one
step of the proof inevitably needs us to touch algebraic $K$-theory as a
space, or more concretely as a simplicial set. We will explain everything we
need for this in a self-contained fashion.

As for \cite{etnclca}, the entire theory should exist without having to assume
that $\mathfrak{A}$ is regular. This will require a slightly different
definition of $\mathsf{LCA}_{\mathfrak{A}}$. Roughly speaking, $\mathsf{LCA}%
_{\mathfrak{A}}$ corresponds to $G$-theory, so just as one has to restrict to
projective modules for $K(\mathfrak{A})$, a similar variation of the
definition of $\mathsf{LCA}_{\mathfrak{A}}$ will work for general
$\mathfrak{A}$. See the introduction of \cite{etnclca} for a few more details
on the r\^{o}le of $G$-theory here.

\begin{acknowledgement}
This work is heavily inspired by Dustin Clausen's paper \cite{clausen}.
Bernhard K\"{o}ck had concretely asked me whether the elements I discussed in
my earlier papers on locally compact modules could be made explicit in the
Nenashev presentation. This article provides a concrete answer. Moreover, I
thank B. Chow, M. Flach, D. Grayson, A. Huber, B. Morin, A. Nickel, M. Wendt,
C. Winges for discussions and/or correspondence. I thank the FRIAS, where the
universal definition of $\left\langle \left\langle P,\varphi,Q\right\rangle
\right\rangle $ was found, for its excellent working conditions and rich
opportunities for academic exchange.
\end{acknowledgement}

\section{The explicit presentations}

\begin{convention}
For us, a ring is always unital and associative. They are not assumed to be
commutative. We freely use the category $\mathsf{LCA}_{\mathfrak{A}}$, see
\cite{etnclca} for its definition and further background.
\end{convention}

We recall the basic design of the two explicit presentations in the format we
shall use.

\subsection{Bass--Swan's presentation}

Let $\varphi:R\rightarrow R^{\prime}$ be a morphism of rings. We will drop the
map $\varphi$ from the notation and simply write $M_{R^{\prime}}:=M\otimes
_{R}R^{\prime}$ for the base change along $\varphi$, where $M$ is an arbitrary
right $R$-module.

Let $\mathsf{Sw}(R,R^{\prime})$ be the following category: Its objects are
triples $(P,\alpha,Q)$, where $P,Q$ are finitely generated projective right
$R$-modules and $\alpha:P_{R^{\prime}}\overset{\sim}{\longrightarrow
}Q_{R^{\prime}}$ is an isomorphism of right $R^{\prime}$-modules. A morphism
$f:(P_{1},\alpha_{1},Q_{1})\rightarrow(P_{2},\alpha_{2},Q_{2})$ is a pair of
right $R$-module homomorphisms $p:P_{1}\rightarrow P_{2}$, $q:Q_{1}\rightarrow
Q_{2}$ such that the diagram%
\[%
\xymatrix{
{P_{1,R^{\prime}}} \ar[r]^{p} \ar[d]^{\sim}_{\alpha_1} & {P_{2,R^{\prime}}}
\ar[d]_{\sim}^{\alpha_2} \\
{Q_{1,R^{\prime}}} \ar[r]_{q} & {Q_{2,R^{\prime}}}
}%
\]
commutes.

Define Bass--Swan's relative $K_{0}$-group, which we denote by
\textquotedblleft$K_{0}(R,R^{\prime})$\textquotedblright, as follows:

\begin{enumerate}
\item It is generated by all objects of $\mathsf{Sw}(R,R^{\prime})$. We write
this as $[P,\alpha,Q]$.

\item \textit{(Relation A)} For morphisms $a,b$ in the category $\mathsf{Sw}%
(R,R^{\prime})$, composable as%
\begin{equation}
(P^{\prime},\alpha^{\prime},Q^{\prime})\overset{a}{\longrightarrow}%
(P,\alpha,Q)\overset{b}{\longrightarrow}(P^{\prime\prime},\alpha^{\prime
\prime},Q^{\prime\prime}) \label{lcimez21a}%
\end{equation}
such that the induced composable arrows $P^{\prime}\hookrightarrow
P\twoheadrightarrow P^{\prime\prime}$ and $Q^{\prime}\hookrightarrow
Q\twoheadrightarrow Q^{\prime\prime}$ are exact sequences of right $R^{\prime
}$-modules, impose the relation:%
\begin{equation}
\lbrack P,\alpha,Q]=[P^{\prime},\alpha^{\prime},Q^{\prime}]+[P^{\prime\prime
},\alpha^{\prime\prime},Q^{\prime\prime}]\text{.} \label{l_Swan_RelationA}%
\end{equation}

\item \textit{(Relation B)} For objects $(P,\alpha,Q)$ and $(Q,\beta,S)$
impose the relation:%
\begin{equation}
\lbrack P,\alpha,Q]+[Q,\beta,S]=[P,\beta\circ\alpha,S]\text{.}
\label{l_Swan_RelationB}%
\end{equation}

\end{enumerate}

\subsection{Nenashev's presentation}

Suppose $\mathsf{C}$ is an exact category. Sherman had the idea that every
element in $K_{1}(\mathsf{C})$ should be expressible in a certain normalized
form. To this end, he used the Gillet--Grayson model for the $K$-theory space
$K(\mathsf{C})$ and then showed that every element $\alpha\in K_{1}%
(\mathsf{C})$, i.e. every closed loop in $\pi_{1}K(\mathsf{C})$, is homotopic
to a loop of a special shape, \cite{MR1621689}. This was extended by Nenashev
to give a complete generator-relator presentation of the group in a series of
papers \cite{MR1409623, MR1621690, MR1637539}.

A \emph{double (short) exact sequence} in $\mathsf{C}$ is the datum of two
short exact sequences%
\[
\mathrm{Yin}:A\overset{p}{\hookrightarrow}B\overset{r}{\twoheadrightarrow
}C\qquad\text{and}\qquad\mathrm{Yang}:A\overset{q}{\hookrightarrow}%
B\overset{s}{\twoheadrightarrow}C\text{,}%
\]
where only the maps may differ, but the objects agree for both the Yin and
Yang exact sequence. We write%
\begin{equation}%
\xymatrix{
A \ar@<1ex>@{^{(}->}[r]^{p} \ar@<-1ex>@{^{(}.>}[r]_{q} & B \ar@<1ex>@{->>}%
[r]^{r} \ar@<-1ex>@{.>>}[r]_{s} & C
}
\label{lcimez8}%
\end{equation}
as a shorthand.

We recall Nenashev's presentation in the concrete form for $\mathsf{LCA}%
_{\mathfrak{A}}$.

\begin{theorem}
\label{thm_NenashevStylePresentation}Let $\mathfrak{A}$ be a regular order in
a finite-dimensional semisimple $\mathbb{Q}$-algebra. Then the abelian group
$K_{1}(\mathsf{LCA}_{\mathfrak{A}})$ has the following presentation:

\begin{enumerate}
\item We attach a generator to each double exact sequence%
\[%
\xymatrix{
A \ar@<1ex>@{^{(}->}[r]^{p} \ar@<-1ex>@{^{(}.>}[r]_{q} & B \ar@<1ex>@{->>}%
[r]^{r} \ar@<-1ex>@{.>>}[r]_{s} & C,
}%
\]
where $A$, $B$ and $C$ are locally compact right $\mathfrak{A}$-modules.

\item Whenever the Yin and Yang exact sequences agree, i.e.,%
\begin{equation}%
\xymatrix{
A \ar@<1ex>@{^{(}->}[r]^{p}_{=} \ar@<-1ex>@{^{(}.>}[r]_{p} & B \ar@
<1ex>@{->>}[r]^{r}_{=} \ar@<-1ex>@{.>>}[r]_{r} & C\text{,}
}
\label{l_C2_NenaAgreeIsZero}%
\end{equation}
we declare the class to vanish.

\item Suppose there is a (not necessarily commutative) $(3\times3)$-diagram%
\[%
\xymatrix@W=0.3in@H=0.3in{
A \ar@<1ex>@{^{(}->}[r] \ar@<1ex>@{^{(}.>}[d] \ar@<-1ex>@{^{(}->}%
[d] \ar@<-1ex>@{^{(}.>}[r] & B \ar@<1ex>@{->>}[r] \ar@<-1ex>@{.>>}%
[r] \ar@<1ex>@{^{(}.>}[d] \ar@<-1ex>@{^{(}->}[d] & C \ar@<1ex>@{^{(}%
.>}[d] \ar@<-1ex>@{^{(}->}[d] \\
D \ar@<1ex>@{^{(}->}[r] \ar@<-1ex>@{^{(}.>}[r] \ar@<1ex>@{.>>}[d] \ar@
<-1ex>@{->>}[d] & E \ar@<1ex>@{->>}[r] \ar@<-1ex>@{.>>}[r] \ar@<1ex>@{.>>}%
[d] \ar@<-1ex>@{->>}[d] & F \ar@<1ex>@{.>>}[d] \ar@<-1ex>@{->>}[d] \\
G \ar@<1ex>@{^{(}->}[r] \ar@<-1ex>@{^{(}.>}[r] & H \ar@<1ex>@{->>}%
[r] \ar@<-1ex>@{.>>}[r] & I, \\
}%
\]
whose rows $Row_{i}$ and columns $Col_{j}$ are double exact sequences. Suppose
after deleting all Yin (resp. all Yang) exact sequences, the remaining diagram
commutes. Then we impose the relation%
\begin{equation}
Row_{1}-Row_{2}+Row_{3}=Col_{1}-Col_{2}+Col_{3}\text{.} \label{l_C_Nenashev}%
\end{equation}

\end{enumerate}
\end{theorem}

We shall write \textquotedblleft$K_{1}(\mathsf{LCA}_{\mathfrak{A}%
})_{\operatorname*{Nenashev}}$\textquotedblright\ if we wish to stress that we
mean the above presentation.

\begin{remark}
[Compatibility with literature]\label{rmk_CompatWithLiterature}Our notational
convention is opposite to the one used in Nenashev's papers \cite{MR1409623,
MR1621690, MR1637539}, and closer instead to the one used in Weibel's
$K$-theory book. In particular, if $\varphi:X\rightarrow X$ is an automorphism
of an object in an exact category $\mathsf{C}$, the canonical map
$\operatorname*{Aut}(X)\rightarrow K_{1}(\mathsf{C})$ maps it to%
\[%
\xymatrix{
0 \ar@<1ex>@{^{(}->}[r] \ar@<-1ex>@{^{(}.>}[r] & X \ar@<1ex>@{->>}%
[r]^{\varphi} \ar@<-1ex>@{.>>}[r]_{1} & X
}%
\text{.}%
\]
See \cite[Equation 2.2]{MR1637539} for a discussion of this. Further, the
vertex $(P,P^{\prime})$ in the Gillet--Grayson model lies in the connected
component of $[P^{\prime}]-[P]\in\pi_{0}K(\mathsf{C})$, i.e. also precisely
opposite to Nenashev's conventions, cf. \cite[\S 1, middle of p.
176]{MR1409625}. In order to be particularly sure about which arrows belong to
the Yin or Yang side, we do not only distinguish by top/bottom or left/right
arrows, but also use solid versus dotted arrows.
\end{remark}

The following is a rather easy exercise in the handling of the Nenashev
presentation, but since it illustrates some basic principles in a simple
fashion, we provide all details.

\begin{lemma}
\label{lemma_DoubeIso}Suppose we are given a double exact sequence%
\[%
\xymatrix{
A^{\prime} \ar@<1ex>@{^{(}->}[r] \ar@<-1ex>@{^{(}.>}[r] & A \ar@
<1ex>@{->>}[r] \ar@<-1ex>@{.>>}[r] & A^{\prime\prime}
}%
\]
and a second one (with $B$s instead of $A$s). Suppose we have a diagram%
\begin{equation}%
\xymatrix{
A^{\prime} \ar[d]^{x^{\prime} } \ar@{^{(}->}[r] & A \ar[d]^{x} \ar@
{->>}[r] & A^{\prime\prime} \ar[d]^{x^{\prime\prime}} \\
B^{\prime} \ar@{^{(}->}[r] & B \ar@{->>}[r] & B^{\prime\prime} \\
}
\label{l424A}%
\end{equation}
where $x^{\prime},x,x^{\prime\prime}$ are isomorphisms, and which commutes
irrespective of whether we use the Yin or Yang exact sequences. Then both
double exact sequences agree in $K_{1}(\mathsf{C})_{\operatorname{Nenashev}}$.
\end{lemma}

\begin{proof}
We have the $(3\times3)$-diagram of double short exact sequences%
\[%
\xymatrix@W=0.3in@H=0.3in{
A^{\prime} \ar@<1ex>@{^{(}->}[r] \ar@<1ex>@{^{(}.>}[d] \ar@<-1ex>@{^{(}%
->}[d] \ar@<-1ex>@{^{(}.>}[r] & A \ar@<1ex>@{->>}[r] \ar@<-1ex>@{.>>}%
[r] \ar@<1ex>@{^{(}.>}[d] \ar@<-1ex>@{^{(}->}[d] & A^{\prime\prime}
\ar@<1ex>@{^{(}.>}[d] \ar@<-1ex>@{^{(}->}[d] \\
B^{\prime} \ar@<1ex>@{^{(}->}[r] \ar@<-1ex>@{^{(}.>}[r] \ar@<1ex>@{.>>}%
[d] \ar@<-1ex>@{->>}[d] & B \ar@<1ex>@{->>}[r] \ar@<-1ex>@{.>>}[r] \ar@
<1ex>@{.>>}[d] \ar@<-1ex>@{->>}[d] & B^{\prime\prime} \ar@<1ex>@{.>>}%
[d] \ar@<-1ex>@{->>}[d] \\
0 \ar@<1ex>@{^{(}->}[r] \ar@<-1ex>@{^{(}.>}[r] & 0 \ar@<1ex>@{->>}%
[r] \ar@<-1ex>@{.>>}[r] & 0 \\
}%
\]
where the double downward arrows \textquotedblleft$\downdownarrows
$\textquotedblright\ in the first row are (both for Yin and Yang) the maps
$x^{\prime},x,x^{\prime\prime}$. Thus, the commutativity of Diagram
\ref{l424A} settles that the top half of the diagram commutes. The downward
arrows \textquotedblleft$\downdownarrows$\textquotedblright\ in the bottom
half all are zero maps. We obtain the relation%
\[
Row_{1}-Row_{2}+Row_{3}=Col_{1}-Col_{2}+Col_{3}%
\]
by Equation \ref{l_C_Nenashev}. We have $[Row_{3}]=0$ since for the bottom row
both Yin and\ Yang sequence agree, and for the same reason $[Col_{i}]=0$ for
$i=1,2,3$. Thus, it remains $[Row_{1}]=[Row_{2}]$, which is precisely our claim.
\end{proof}

\begin{lemma}
\label{lemma_leftrightswap}We have $\left[
\xymatrix{
X \ar@<1ex>@{^{(}->}[r]^{1} \ar@<-1ex>@{^{(}.>}[r]_{\varphi} & X \ar@
<1ex>@{->>}[r] \ar@<-1ex>@{.>>}[r] & 0
}%
\right]  =\left[
\xymatrix{
0 \ar@<1ex>@{^{(}->}[r] \ar@<-1ex>@{^{(}.>}[r] & X \ar@<1ex>@{->>}%
[r]^{\varphi} \ar@<-1ex>@{.>>}[r]_{1} & X
}%
\right]  $.
\end{lemma}

\begin{proof}
By Remark \ref{rmk_CompatWithLiterature} the left term agrees with $\tilde
{l}(\varphi)$ in \cite[Equation 2.2]{MR1637539}, and thus with $l(\varphi)$ by
Lemma 3.1, \textit{loc. cit.} Translating back with Remark
\ref{rmk_CompatWithLiterature} yields our claim.
\end{proof}

\subsection{Graphical schematics for double exact
sequences\label{sect_Graphics}}

The notation of Equation \ref{lcimez8} is not very helpful when the involved
objects and morphisms are complicated.

In addition to the standard notation, we shall introduce a slightly different
notation in the present text: Suppose we write the graphical schematic%
\begin{equation}%
\begin{tabular}
[c]{ccccc}%
$A_{1}^{\prime}$ & $\overset{a_{1}^{\prime}}{\hookrightarrow}$ & $A_{1}$ &
$\overset{a_{1}}{\twoheadrightarrow}$ & $A_{1}^{\prime\prime}$\\
$A_{2}^{\prime}$ & $\overset{a_{2}^{\prime}}{\hookrightarrow}$ & $A_{2}$ &
$\overset{a_{2}}{\twoheadrightarrow}$ & $A_{2}^{\prime\prime}$\\
$\vdots$ &  & $\vdots$ &  & $\vdots$\\\hline
$B_{1}^{\prime}$ & $\overset{b_{1}^{\prime}}{\hookrightarrow}$ & $B_{1}$ &
$\overset{b_{1}}{\twoheadrightarrow}$ & $B_{1}^{\prime\prime}$\\
$B_{2}^{\prime}$ & $\overset{b_{2}^{\prime}}{\hookrightarrow}$ & $B_{2}$ &
$\overset{b_{2}}{\twoheadrightarrow}$ & $B_{1}^{\prime\prime}$\\
$\vdots$ &  & $\vdots$ &  & $\vdots$%
\end{tabular}
\ \ \label{l_Schematic}%
\end{equation}
and suppose further

\begin{enumerate}
\item each $A_{i}^{\prime}\hookrightarrow A_{i}\twoheadrightarrow
A_{i}^{\prime\prime}$ is an exact sequence,

\item each $B_{i}^{\prime}\hookrightarrow B_{i}\twoheadrightarrow
B_{i}^{\prime\prime}$ is an exact sequence,

\item in each column (left, middle, right) the objects above and below the
horizontal delimiter are isomorphic by a concrete isomorphism, i.e.%
\begin{align}
I^{\prime}  &  :\bigoplus_{i}A_{i}^{\prime}\overset{\sim}{\longrightarrow
}\bigoplus_{i}B_{i}^{\prime}\label{l_MapsI}\\
I  &  :\bigoplus_{i}A_{i}\overset{\sim}{\longrightarrow}\bigoplus_{i}%
B_{i}\nonumber\\
I^{\prime\prime}  &  :\bigoplus_{i}A_{i}^{\prime\prime}\overset{\sim
}{\longrightarrow}\bigoplus_{i}B_{i}^{\prime\prime}\text{,}\nonumber
\end{align}

\end{enumerate}

then we attach the following double short exact sequence to this datum:%
\[
\kappa=\left[
\xymatrix{
{\bigoplus A_{i}^{\prime}} \ar@<1ex>@{^{(}->}[r]^{p} \ar@<-1ex>@{^{(}%
.>}[r]_{q} & {\bigoplus A_{i}} \ar@<1ex>@{->>}[r]^{r} \ar@<-1ex>@{.>>}[r]_{s}
& {\bigoplus A_{i}^{\prime\prime}}
}%
\right]  \text{,}%
\]
where

\begin{enumerate}
\item $p:=\bigoplus a_{i}^{\prime}$, $r:=\bigoplus a_{i}$,

\item $q:=I^{-1}\circ\left(  \bigoplus b_{i}^{\prime}\right)  \circ I^{\prime
}$, $s:=I^{\prime\prime-1}\circ\left(  \bigoplus b_{i}\right)  \circ I$.
\end{enumerate}

\textit{(Wiring)} We tacitly assume here that $i$ runs through finitely many
values only, i.e. our input graphic schematic will always only have finitely
many rows. We call the datum of $I^{\prime},I,I^{\prime\prime}$ the
\emph{wiring} of the schematic.

We note that since each $A_{i}^{\prime}\hookrightarrow A_{i}\twoheadrightarrow
A_{i}^{\prime\prime}$ is an exact sequence, so is their direct sum, and
analogously for $B$. Thus, $\kappa$ is indeed a double short exact sequence.

\section{The comparison map\label{sect_ExplicitMap}}

Suppose $\mathfrak{A}$ is an arbitrary order in a finite-dimensional
semisimple $\mathbb{Q}$-algebra $A$. In this section we shall set up a
concrete map from the Bass--Swan to the Nenashev presentation.\medskip

We recall: If $P$ denotes a finitely generated right $\mathfrak{A}$-module, we
write $P_{\mathbb{R}}$ to denote the base change $P\otimes_{\mathfrak{A}%
}\mathbb{R}$ (or equivalently $P\otimes_{\mathfrak{A}}A_{\mathbb{R}}$). For
every $P$, we have a canonical short exact sequence%
\[
P\overset{\iota_{P}}{\hookrightarrow}P_{\mathbb{R}}\overset{q_{P}%
}{\twoheadrightarrow}T_{P}\qquad\text{in}\qquad\mathsf{LCA}_{\mathfrak{A}%
}\text{,}%
\]
where $P$ refers to itself, equipped with the discrete topology,
$P_{\mathbb{R}}$ is equipped with the real vector space topology, and $T_{P}$
denotes the quotient taken in $\mathsf{LCA}_{\mathfrak{A}}$. This means that
the underlying topological space of $T_{P}$ is a real torus $\mathbb{T}^{n}$
with $n:=\dim_{\mathbb{R}}(P_{\mathbb{R}})$. The key point is that,
topologically, $P$ is a discrete full rank lattice in $P_{\mathbb{R}}$.

Now consider the generator $[P,\varphi,Q]$ in Bass--Swan's presentation. Then
$\varphi:P_{\mathbb{R}}\overset{\sim}{\rightarrow}Q_{\mathbb{R}}$ is an
isomorphism. Note that alongside%
\begin{equation}
Q\overset{\iota_{Q}}{\hookrightarrow}Q_{\mathbb{R}}\overset{q_{Q}%
}{\twoheadrightarrow}T_{Q}\label{l_siup2}%
\end{equation}
we also get the short exact sequence%
\begin{equation}
Q\overset{\varphi^{-1}\circ\iota_{Q}}{\hookrightarrow}P_{\mathbb{R}}%
\overset{q_{Q}\circ\varphi}{\twoheadrightarrow}T_{Q}\text{.}\label{l_siup2a}%
\end{equation}
It is basically the same, except for that we have replaced the middle object
via the isomorphism $\varphi$. This sequence will henceforth play an important r\^{o}le.

We now make the following crucial definition (and explain some potentially
unclear notation below the definition):

\begin{definition}
Let $(P,\alpha,Q)$ be an object in the category $\mathsf{Sw}(\mathfrak{A}%
,A_{\mathbb{R}})$. We write $\left\langle \left\langle P,\alpha,Q\right\rangle
\right\rangle $ for the following schematic:%
\begin{equation}%
\begin{tabular}
[c]{ccccc}%
$0$ & $\longrightarrow$ & $\bigoplus P$ & $\overset{1}{\longrightarrow}$ &
$\bigoplus P$\\
$P$ & $\longrightarrow$ & $P_{\mathbb{R}}$ & $\longrightarrow$ & $T_{P}$\\
$\prod T_{P}$ & $\overset{1}{\longrightarrow}$ & $\prod T_{P}$ &
$\longrightarrow$ & $0$\\
$Q$ & $\longrightarrow$ & $\bigoplus Q$ & $\overset{s}{\longrightarrow}$ &
$\bigoplus Q$\\
$\prod T_{Q}$ & $\overset{s}{\longrightarrow}$ & $\prod T_{Q}$ &
$\longrightarrow$ & $T_{Q}$\\\hline
$P$ & $\longrightarrow$ & $\bigoplus P$ & $\overset{s}{\longrightarrow}$ &
$\bigoplus P$\\
$\prod T_{P}$ & $\overset{s}{\longrightarrow}$ & $\prod T_{P}$ &
$\longrightarrow$ & $T_{P}$\\
$0$ & $\longrightarrow$ & $\bigoplus Q$ & $\overset{1}{\longrightarrow}$ &
$\bigoplus Q$\\
$Q$ & $\overset{\varphi^{-1}\circ\iota_{Q}}{\longrightarrow}$ & $P_{\mathbb{R}%
}$ & $\overset{q_{Q}\circ\varphi}{\longrightarrow}$ & $T_{Q}$\\
$\prod T_{Q}$ & $\overset{1}{\longrightarrow}$ & $\prod T_{Q}$ &
$\longrightarrow$ & $0$%
\end{tabular}
\ \ \label{l_S1}%
\end{equation}
Whenever convenient, we also use the notation $\left\langle \left\langle
P,\alpha,Q\right\rangle \right\rangle $ for the associated double short exact
sequence (following the recipe of \S \ref{sect_Graphics}).
\end{definition}

We have used some shorthands here, which we explain now:

\begin{enumerate}
\item The symbol $\bigoplus P$ refers to the coproduct indexed over
$\mathbb{Z}_{\geq0}$. By the map \textquotedblleft$s$\textquotedblright\ (as
in `shift') in $P\hookrightarrow\bigoplus P\overset{s}{\twoheadrightarrow
}\bigoplus P$ we mean the map%
\begin{equation}
(p_{0},p_{1},p_{2},\ldots)\mapsto(p_{1},p_{2},\ldots)\text{.}
\label{l_Map_Shift}%
\end{equation}
Analogously, for $Q$.

\item The symbol $\prod T_{P}$ refers to the product indexed over
$\mathbb{Z}_{\geq0}$. By the map \textquotedblleft$s$\textquotedblright%
\ (again as in `shift') in $\prod T_{P}\overset{s}{\hookrightarrow}\prod
T_{P}\twoheadrightarrow T_{P}$ we mean the map
\[
(t_{0},t_{1},t_{2},\ldots)\mapsto(0,t_{0},t_{1},\ldots)\text{.}%
\]

\item We observe that in the left column the objects above the delimiter line
are the same ones as below the line, just with the zero object placed in a
different location. Thus, for the map $I^{\prime}$ in Equation \ref{l_MapsI}
we take the obvious map. The same is true for the right column, so for
$I^{\prime\prime}$ we also take the obvious map. The middle column is a little
more involved since there is also a permutation of the summands involved; but
again it is clear what map we take for $I$.
\end{enumerate}

We note that $\bigoplus P$ carries the discrete topology, and $\prod T_{P}$ is
compact by Tychonoff's theorem. We call both shift maps $s$, although they are
a little different, because the underlying idea of both maps is so similar.

We now have fully defined $\left\langle \left\langle P,\alpha,Q\right\rangle
\right\rangle $. Still, let us describe it a little more. The very same object
can also be drawn as follows:%
\begin{equation}%
{\includegraphics[
height=1.855in,
width=2.4154in
]%
{gfx_1-eps-converted-to.pdf}%
}
\label{l_S2}%
\end{equation}
Here again the three columns correspond to the left, middle and right object
in the double exact sequence. The corresponding left (resp. middle, resp.
right) object in the double exact sequence arises by taking the direct sum of
the left (resp. middle, resp. right) column. The solid arrows correspond to
the Yin arrows, the dotted ones to Yang. This depiction is probably the best
to see the structure of the wiring.

To make it more readable, we have dropped the zero objects in the above
picture and neglected labeling the arrows in the cases where it is clear:
among the $\bigoplus P$ and $\prod T_{P}$, resp. for $Q$, the horizontal exact
sequences stem from the shift maps $s$, while the single arrows correspond to
the identity map. Thus, all in all, $\left\langle \left\langle P,\alpha
,Q\right\rangle \right\rangle $ defines a double exact sequence of the shape%
\[
\left[
\xymatrix{
{\prod T_{P}\oplus P\oplus Q\oplus\prod T_{Q}} \ar@<1ex>@{^{(}->}%
[r] \ar@<-1ex>@{^{(}.>}[r] & {\prod T_{P}\oplus P\oplus P_{\mathbb{R}}\oplus
Q\oplus\prod T_{Q}} \ar@<1ex>@{->>}[r] \ar@<-1ex>@{.>>}[r] & {T_{P}%
\oplus\bigoplus P\oplus\bigoplus Q\oplus T_{Q}}
}%
\right]  \text{,}%
\]
where the morphisms are slightly involved to spell out, and surely easier to
understand from\ Figure \ref{l_S1} or Figure \ref{l_S2} than by trying to
squeeze them into the double exact sequence notation. It would lead to very
heavy arrow labels.

\begin{theorem}
\label{thm_VarThetaWellDefined}Suppose $\mathfrak{A}$ is an arbitrary order in
a finite-dimensional semisimple $\mathbb{Q}$-algebra $A$. The map%
\begin{align*}
\vartheta:K_{0}(\mathfrak{A},\mathbb{R}) &  \longrightarrow K_{1}%
(\mathsf{LCA}_{\mathfrak{A}})_{\operatorname*{Nenashev}}\\
\lbrack P,\varphi,Q] &  \longmapsto\text{\emph{double exact sequence of }%
}\left\langle \left\langle P,\varphi,Q\right\rangle \right\rangle
\end{align*}
is well-defined.
\end{theorem}

We shall split the slightly involved proof into several little lemmata.

\begin{lemma}
\label{sv3}The map $\vartheta$ respects Relation A (Equation
\ref{l_Swan_RelationA}), i.e. for morphisms $a,b$ composable as%
\[
(P^{\prime},\alpha^{\prime},Q^{\prime})\overset{a}{\longrightarrow}%
(P,\alpha,Q)\overset{b}{\longrightarrow}(P^{\prime\prime},\alpha^{\prime
\prime},Q^{\prime\prime})
\]
such that the induced composable arrows $P^{\prime}\hookrightarrow
P\twoheadrightarrow P^{\prime\prime}$ and $Q^{\prime}\hookrightarrow
Q\twoheadrightarrow Q^{\prime\prime}$ are exact sequences of right $R^{\prime
}$-modules, we have%
\begin{equation}
\lbrack\left\langle \left\langle P,\alpha,Q\right\rangle \right\rangle
]=[\left\langle \left\langle P^{\prime},\alpha^{\prime},Q^{\prime
}\right\rangle \right\rangle ]+[\left\langle \left\langle P^{\prime\prime
},\alpha^{\prime\prime},Q^{\prime\prime}\right\rangle \right\rangle
]\text{.}\label{lmiur11}%
\end{equation}

\end{lemma}

\begin{proof}
First of all, we note that the input datum gives rise to a commutative diagram
of the shape%
\begin{equation}%
\xymatrix{
P^{\prime}_{\mathbb{R} } \ar@{^{(}->}[r] \ar[d]^{\alpha^{\prime} }
& P_{\mathbb{R} } \ar@{->>}[r] \ar[d]^{\alpha} & P^{\prime\prime}_{\mathbb{R}
} \ar[d]^{\alpha^{\prime\prime} } \\
Q^{\prime}_{\mathbb{R} } \ar@{^{(}->}[r] & Q_{\mathbb{R} } \ar@{->>}%
[r] & Q^{\prime\prime}_{\mathbb{R} } \text{.} \\
}
\label{lmiur10}%
\end{equation}
We note that the exactness of $P^{\prime}\hookrightarrow P\twoheadrightarrow
P^{\prime\prime}$ also gives natural exact sequences
\begin{equation}
\bigoplus P^{\prime}\hookrightarrow\bigoplus P\twoheadrightarrow\bigoplus
P^{\prime\prime}\text{,}\qquad\text{and}\qquad\prod T_{P^{\prime}%
}\hookrightarrow\prod T_{P}\twoheadrightarrow\prod T_{P^{\prime\prime}%
}\text{.} \label{lmiur10_W}%
\end{equation}
The same is true for $Q$, $Q_{\mathbb{R}}$, $\bigoplus Q$, $\prod T_{Q}$. This
means that we get short exact sequences for all the entries in $\left\langle
\left\langle P,\alpha,Q\right\rangle \right\rangle $, having their
counterparts in $\left\langle \left\langle P^{\prime},\alpha^{\prime
},Q^{\prime}\right\rangle \right\rangle $ and $\left\langle \left\langle
P^{\prime\prime},\alpha^{\prime\prime},Q^{\prime\prime}\right\rangle
\right\rangle $ on the left resp. right side. We thus can set up a
$(3\times3)$-diagram%
\[%
\begin{array}
[c]{c}%
\left\langle \left\langle P^{\prime},\alpha^{\prime},Q^{\prime}\right\rangle
\right\rangle \\
\downdownarrows\\
\left\langle \left\langle P,\alpha,Q\right\rangle \right\rangle \\
\downdownarrows\\
\left\langle \left\langle P^{\prime\prime},\alpha^{\prime\prime}%
,Q^{\prime\prime}\right\rangle \right\rangle
\end{array}
\]
as follows: The columns are induced for both Yin and Yang arrows from the
sequences above, e.g., in Equation \ref{lmiur10_W}, and correspondingly for
all other entries in the schematic. On most terms in the schematic it is
obvious that the resulting squares commute. The only potentially delicate
piece is the middle cross for all three terms of the exact sequences (by
middle cross we mean the four diagonal arrows in Figure \ref{l_S2}). We just
discuss the monic piece around the `$\hookrightarrow$' in detail; the epic
piece can be treated analogously. We depict the situation below. On the left
we see the middle cross of $\left\langle \left\langle P^{\prime}%
,\alpha^{\prime},Q^{\prime}\right\rangle \right\rangle $, and on the right the
middle cross of $\left\langle \left\langle P,\alpha,Q\right\rangle
\right\rangle $. Let us describe the compatible maps from left to right which
then form the top downward arrows \textquotedblleft$\downdownarrows
$\textquotedblright\ in the above $(3\times3)$-diagram.
\begin{equation}%
{\includegraphics[
height=1.8351in,
width=4.7651in
]%
{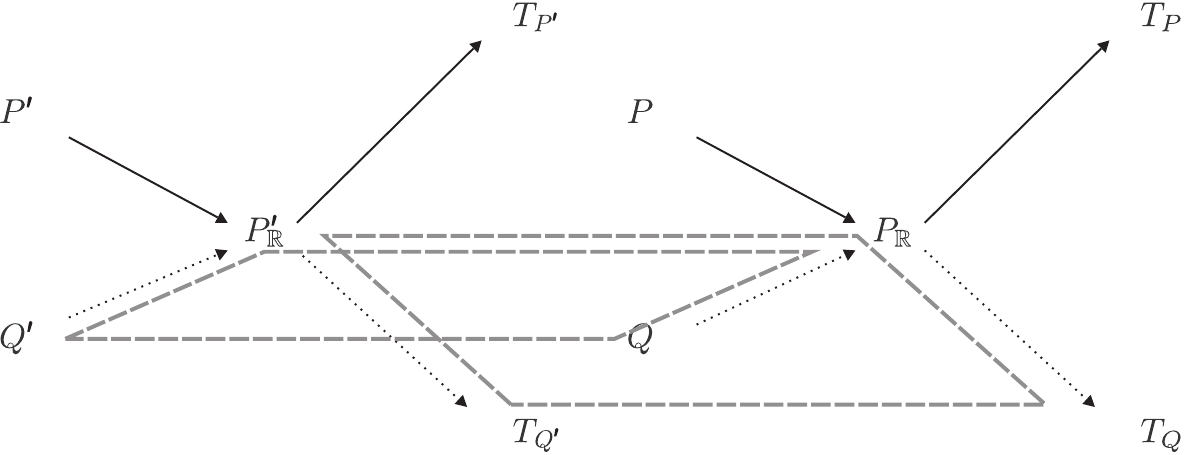}%
}
\label{l_siup1}%
\end{equation}
On the Yin side (the solid arrows), we just use the natural maps coming from
Equation \ref{l_siup2}, e.g.,%
\[
P^{\prime}\hookrightarrow P_{\mathbb{R}}^{\prime}\twoheadrightarrow
T_{P^{\prime}}%
\]
on the left. On the Yang side, we could do the same if we had $Q_{\mathbb{R}}$
sitting in the middle, since we have a corresponding natural sequence%
\[
Q^{\prime}\hookrightarrow Q_{\mathbb{R}}^{\prime}\twoheadrightarrow
T_{Q^{\prime}}\text{.}%
\]
Thus, we play the trick of Equation \ref{l_siup2a}, that is: We pre- and
post-compose $Q_{\mathbb{R}}^{\prime}$ with the isomorphisms $\alpha^{\prime}%
$, $\alpha^{\prime-1}$; and correspondingly for $Q_{\mathbb{R}}$ with
$\alpha,\alpha^{-1}$. The necessary diagram of commutativities to check that
this compatibly fills the squares in Figure \ref{l_siup1} is
\textit{precisely} the diagram in Equation \ref{lmiur10}. In the figure above,
the front dashed contour parallelogram (the one more to the right) corresponds
precisely to the left commuting square in Equation \ref{lmiur10}, composed
with the natural maps $q_{Q^{\prime}}$ resp. $q_{Q}$. The back dashed contour
parallelogram (the one more to the left) corresponds to the same square, but
for the inverse morphisms $\alpha^{\prime-1}$ and $\alpha^{-1}$ instead, this
time pre-composed with the natural maps $\iota_{Q^{\prime}}$ resp. $\iota_{Q}%
$. Since in the columns we use the same maps for the Yin and Yang side, in
Nenashev's relation of the $(3\times3)$-diagram we have no column
contributions, and the remaining relation among the rows in precisely Equation
\ref{lmiur11}, proving our claim.
\end{proof}

\begin{lemma}
\label{Lemma_SwapIsZeroInLCAA}Let $X\in\mathsf{LCA}_{\mathfrak{A}}$ be
arbitrary. In $K_{1}(\mathsf{LCA}_{\mathfrak{A}})_{\operatorname*{Nenashev}}$
the class of%
\begin{equation}
s_{X}:=\left[
\xymatrix{
X \oplus X \ar@<1ex>@{^{(}->}[r]^{1} \ar@<-1ex>@{^{(}.>}[r]_{q} & X \oplus
X \ar@<1ex>@{->>}[r] \ar@<-1ex>@{.>>}[r] & 0
}%
\right]  \text{,} \label{lmiur2}%
\end{equation}
where $q$ swaps both summands, i.e. $(x,y)\mapsto(y,x)$, is zero. The same is
true if we additionally impose a sign switch, i.e. use $(x,y)\mapsto(y,-x)$ instead.
\end{lemma}

\begin{proof}
We deal with the case without the sign: (Step 1)\ Suppose $X$ is a vector
$\mathfrak{A}$-module. Then it can be written as $X=\mathfrak{X}%
\otimes_{\mathbb{Z}}\mathbb{R}$, where $\mathfrak{X}$ is a discrete right
$\mathfrak{A}$-module and full $\mathbb{Z}$-rank lattice in $X$. We can also
consider the swap for $\mathfrak{X}\oplus\mathfrak{X}\rightarrow
\mathfrak{X}\oplus\mathfrak{X}$ in $\operatorname*{PMod}(\mathfrak{A})$. Note
that we can send this swapping automorphism along the exact functors%
\[
\operatorname*{PMod}(\mathfrak{A})\longrightarrow\operatorname*{PMod}%
(A_{\mathbb{R}})\longrightarrow\mathsf{LCA}_{\mathfrak{A}}\text{,}%
\]
which are $(-)\mapsto(-)\otimes_{\mathbb{Z}}\mathbb{R}$, and then interpreting
the vector space as its underlying locally compact right $\mathfrak{A}$-module
with the real topology. The resulting automorphism in $\mathsf{LCA}%
_{\mathfrak{A}}$ is precisely the one of our claim. However, the composition
$K_{1}(\mathfrak{A})\rightarrow K_{1}(A_{\mathbb{R}})\rightarrow
K_{1}(\mathsf{LCA}_{\mathfrak{A}})$ is zero, which is seen in terms of the
Nenashev presentation in \cite[Example 2.7]{etnclca}. At least if
$\mathfrak{A}$ is regular, it alternatively follows from \cite[Theorem
11.2]{etnclca}, which avoids working in the Nenashev presentation. The proof
given loc. cit. would also extend to arbitrary orders $\mathfrak{A}$. (Step
2)\ Now let $X\in\mathsf{LCA}_{\mathfrak{A}}$ be arbitrary. By \cite[Lemma
6.5]{etnclca} we can find an exact sequence%
\begin{equation}
C\hookrightarrow X\twoheadrightarrow V\oplus D \label{lmiur1}%
\end{equation}
in $\mathsf{LCA}_{\mathfrak{A}}$ with $C$ compact, $D$ discrete and $V$ a
vector $\mathfrak{A}$-module. Consider the double exact sequence%
\[
S=\left[
\xymatrix{
C \oplus C \ar@<1ex>@{^{(}->}[r] \ar@<-1ex>@{^{(}.>}[r] & X \oplus
X \ar@<1ex>@{->>}[r] \ar@<-1ex>@{.>>}[r] & (V \oplus D) \oplus(V \oplus D)
}%
\right]
\]
which arises as the direct sum of two copies of Equation \ref{lmiur1}.
Consider the $(3\times3)$-diagram%
\[%
\begin{array}
[c]{c}%
S\\
\downdownarrows\\
S\\
\downdownarrows\\
0\text{,}%
\end{array}
\]
where the top downward arrows \textquotedblleft$\downdownarrows$%
\textquotedblright\ are the identity on the Yin side, and\ are swapping the
respective two summands on the Yang side. The resulting relation%
\[
Row_{1}-Row_{2}+Row_{3}=Col_{1}-Col_{2}+Col_{3}%
\]
simplifies to $[S]-[S]+[0]=[s_{C}]-[s_{X}]+[s_{V\oplus D}]$, where $s_{(-)}$
denotes the class of Equation \ref{lmiur2} for $C$, $X$ and $V\oplus D$
(instead of $X$) respectively. Similarly we obtain $[s_{V\oplus D}%
]=[s_{V}]+[s_{D}]$. We have $[s_{V}]=0$ by Step 1. The classes of $[s_{C}]$
and $[s_{D}]$ can be seen to be zero by an Eilenberg swindle, see again the
method of \cite[Example 2.7]{etnclca}. Alternatively: For the compact object
$C$, we may regard $C$ in $\mathsf{LCA}_{\mathfrak{A},C}$ and consider the
class of $s_{C}$ there. Thus, our $[s_{C}]$ is the image in $K_{1}%
(\mathsf{LCA}_{\mathfrak{A}})$ under the exact functor $\mathsf{LCA}%
_{\mathfrak{A},C}\rightarrow\mathsf{LCA}_{\mathfrak{A}}$. Since $K_{1}%
(\mathsf{LCA}_{\mathfrak{A},C})=0$ by the Eilenberg swindle (\cite[Lemma
4.2]{obloc}; the category $\mathsf{LCA}_{\mathfrak{A},C}$ is closed under
infinite products by Tychonoff's Theorem), it follows that $[s_{C}]=0$.
Analogously, we obtain $[s_{D}]=0$ by using the corresponding functor from
discrete $\mathfrak{A}$-modules, an exact category which is closed under
infinite coproducts. Combining these equations, $[s_{X}]=0$, which is what we
wanted to show. For the signed swap, the same argument works; however, the
signed swap is in fact trivial in $K$-theory generally, see \cite[Lemma 3.2,
(i)]{MR1637539} (this is valid, recall how to translate notation, Remark
\ref{rmk_CompatWithLiterature}).
\end{proof}

\begin{remark}
\label{rmk_InversionIsTrivial}The above lemma might sound counter-intuitive,
especially if $X$ is a vector $\mathfrak{A}$-module. Its statement would be
false in $K_{1}(\mathbb{R})$. What it really says is that $K$-theory does not
see orientation-reversal in $\mathsf{LCA}_{\mathfrak{A}}$. To see this, note
that in the above proof, the swapping class, once viewed in $K_{1}%
(\mathfrak{A})$, agrees via%
\[%
\begin{pmatrix}
& 1\\
1 &
\end{pmatrix}
=%
\begin{pmatrix}
1 & 1\\
& 1
\end{pmatrix}%
\begin{pmatrix}
1 & \\
-1 & 1
\end{pmatrix}%
\begin{pmatrix}
1 & 1\\
& 1
\end{pmatrix}%
\begin{pmatrix}
-1 & \\
& 1
\end{pmatrix}
\]
with the class of multiplication by $-1$. Note that the first three matrices
are elementary and thus zero in $K_{1}$. This observation extends the
invisibility of orientations to the Haar measure as discussed in
\cite[\S 2]{etnclca}, and when $\mathfrak{A}=\mathbb{Z}$, it is literally that.
\end{remark}

Relation B (Equation \ref{l_Swan_RelationB}) is a little more complicated to
handle. It is already not completely obvious that $[P,\operatorname*{id},P]$
is being sent to zero. Hence, as a warm-up and illustration of the general
method, let us prove this first.

\begin{lemma}
\label{sv1}The map $\vartheta$ sends $[P,\operatorname*{id},P]$ to zero.
\end{lemma}

\begin{proof}
Unravelling definitions, we find that $\left\langle \left\langle
P,\operatorname*{id},P\right\rangle \right\rangle $ corresponds to the double
exact sequence depicted below on the left:%
\begin{equation}%
{\includegraphics[
height=1.8593in,
width=5.4051in
]%
{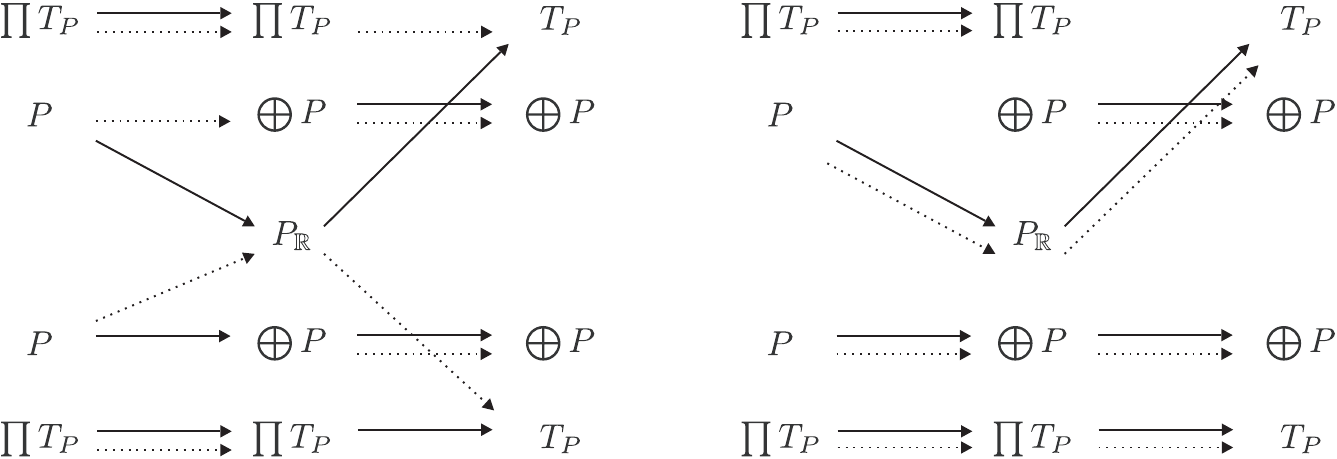}%
}
\label{lsui1}%
\end{equation}
Next, we produce a morphism of double exact sequences from the left schematic
to the right. Objectwise, the maps are: (a) the identity on the Yin side
everywhere, (b) and on the Yang side, we use the swap map%
\begin{align}
\prod T_{P}\oplus\prod T_{P}  &  \longrightarrow\prod T_{P}\oplus\prod
T_{P}\label{lseem1}\\
(x,y)  &  \longmapsto(y,x)\nonumber
\end{align}
and analogously we swap the summands of $P\oplus P$, as well as of $\bigoplus
P\oplus\bigoplus P$. We use the identity map on $P_{\mathbb{R}}$. Let us
explain this along Figure \ref{lsui1}: One can think about the map which we
have just described in two ways: Firstly, as exchanging all objects $\prod
T_{P}$, $P,$ $\bigoplus P$ on the upper half of the figure with their
counterparts on the lower half; the object $P_{\mathbb{R}}$ remains where it
is. When we speak of the top and lower half here, we do not mean the Yin and
Yang side (which often in double exact sequences are depicted as the top and
bottom arrow), but we rather just mean the obvious symmetry of Figure
\ref{lsui1} when mirroring it along the horizontal middle axis. Our swapping
operation only refers to this symmetry and only to the Yang side, i.e. the
dotted arrows. However, in the above figure on the right, we use a different
graphical presentation: Instead of swapping the objects, we leave all the
objects at precisely the same position as before. Instead, we swap all the
dotted (i.e. Yang) arrows between them. So, one way to think about going from
left to right is that all the dotted arrows $(A)\dashrightarrow(B)$ on the
upper half get exchanged with their corresponding arrow $(A)\dashrightarrow
(B)$ on the lower half. In total, this morphism of double exact sequences
gives a $(3\times3)$-diagram%
\[%
\begin{array}
[c]{c}%
\left\langle \left\langle P,\operatorname*{id},P\right\rangle \right\rangle \\
\downdownarrows\\
S\\
\downdownarrows\\
0\text{,}%
\end{array}
\]
where $S$ denotes the double exact sequence on the right in Figure
\ref{lsui1}, the downward arrows \textquotedblleft$\downdownarrows
$\textquotedblright\ are those induced from $\operatorname*{id}:P_{\mathbb{R}%
}\rightarrow P_{\mathbb{R}}$ (for both Yin and Yang), the identity on all
objects on the Yin side, and the swapping maps of Equation \ref{lseem1} for
all other objects on the Yang side. Since these maps are all isomorphisms, we
get the zero double exact sequence as the quotient. The resulting Nenashev
relation%
\[
Row_{1}-Row_{2}+Row_{3}=Col_{1}-Col_{2}+Col_{3}%
\]
of Equation \ref{l_C_Nenashev} simplifies to $[\left\langle \left\langle
P,\operatorname*{id},P\right\rangle \right\rangle ]-[S]+[0]=0$, because
$[Col_{i}]=0$ for $i=1,2,3$. The latter is seen as follows: The column classes
spelled out (and written horizontally to save space) are direct sums of double
exact sequences%
\[
\left[
\xymatrix{
(-) \ar@<1ex>@{^{(}->}[r]^{1} \ar@<-1ex>@{^{(}.>}[r]_{q} & (-) \ar@
<1ex>@{->>}[r] \ar@<-1ex>@{.>>}[r] & 0
}%
\right]  \text{,}%
\]
where $q$ is either the identity or a swapping map. Thus, its class is always
zero, either by the relation of Equation \ref{l_C2_NenaAgreeIsZero} or by
Lemma \ref{Lemma_SwapIsZeroInLCAA}.
\end{proof}

With this preparation, we are ready for a generalization of the previous lemma.

\begin{lemma}
\label{sv2}The map $\vartheta$ respects Relation B (Equation
\ref{l_Swan_RelationB}), i.e. we have%
\[
\lbrack\left\langle \left\langle P,\psi\varphi,R\right\rangle \right\rangle
]=[\left\langle \left\langle P,\varphi,Q\right\rangle \right\rangle
]+[\left\langle \left\langle Q,\psi,R\right\rangle \right\rangle ]\text{.}%
\]

\end{lemma}

\begin{proof}
We consider the graphical schematics of $\left\langle \left\langle
P,\varphi,Q\right\rangle \right\rangle $ and $\left\langle \left\langle
Q,\psi,R\right\rangle \right\rangle $, as they were given in Equation
\ref{l_S1}. Using this data as input, we form a new schematic: Write the top
half lines of both $\left\langle \left\langle P,\varphi,Q\right\rangle
\right\rangle $ and $\left\langle \left\langle Q,\psi,R\right\rangle
\right\rangle $ under each other, and analogously for the bottom lines of
each, giving the new schematic which we call $J$. Along with its wiring, it
can be depicted as follows below on the left:%
\begin{equation}%
{\includegraphics[
height=3.9583in,
width=2.7086in
]%
{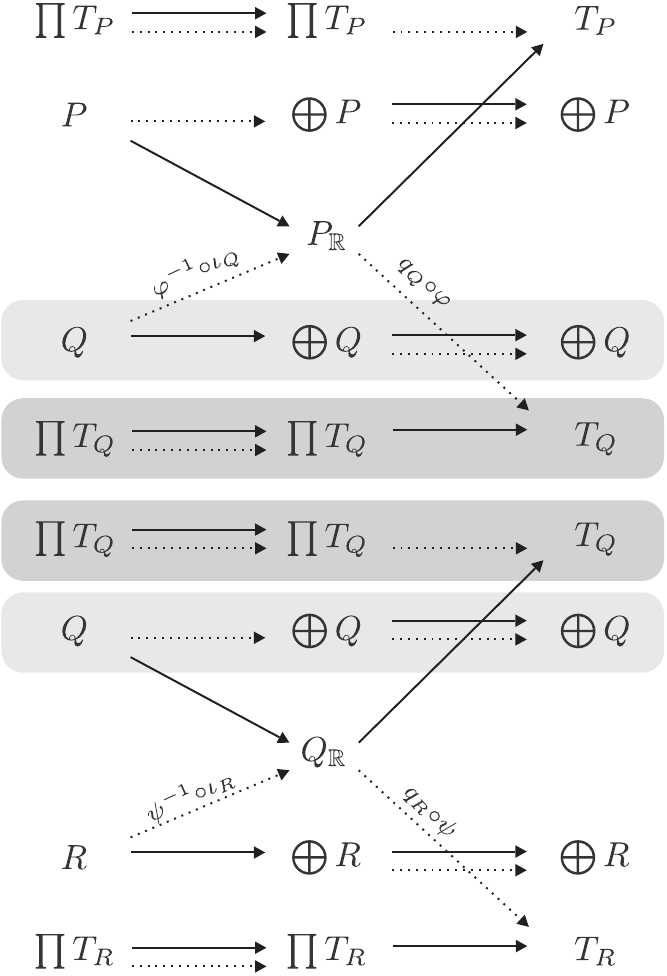}%
}
\qquad%
{\includegraphics[
height=3.9583in,
width=2.7086in
]%
{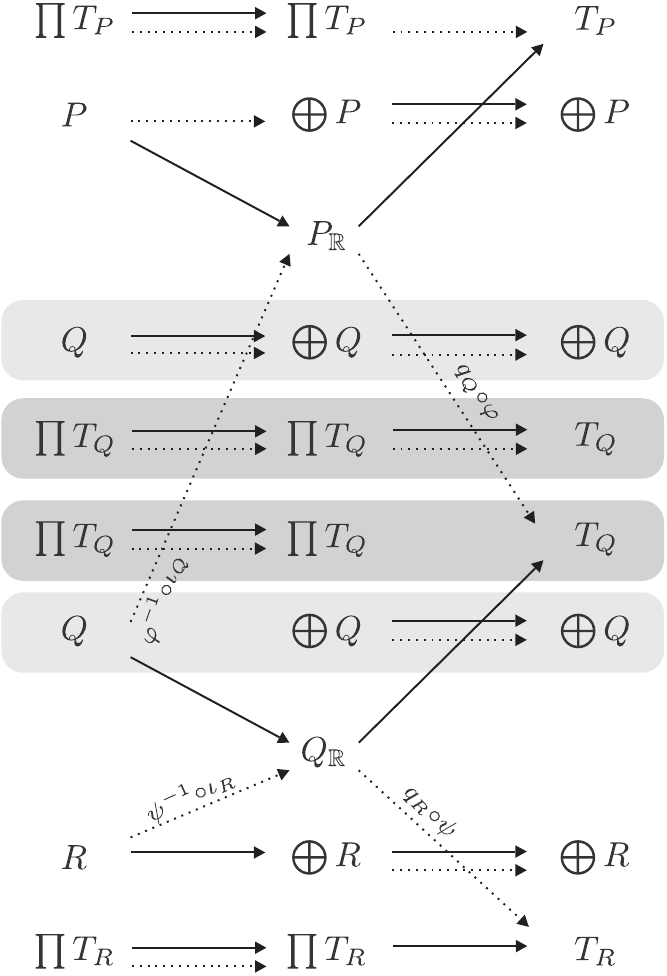}%
}
\label{lmiur5}%
\end{equation}
and we ignore the shaded areas for the moment. This schematic defines a double
exact sequence. We find a natural $(3\times3)$-diagram%
\[%
\begin{array}
[c]{c}%
\left\langle \left\langle P,\varphi,Q\right\rangle \right\rangle \\
\downdownarrows\\
J\\
\downdownarrows\\
\left\langle \left\langle Q,\psi,R\right\rangle \right\rangle \text{,}%
\end{array}
\]
where the top downward arrows \textquotedblleft$\downdownarrows$%
\textquotedblright\ are (both for Yin and Yang) just the inclusion of the
$\left\langle \left\langle P,\varphi,Q\right\rangle \right\rangle
$-subschematic, and correspondingly the bottom arrows \textquotedblleft%
$\downdownarrows$\textquotedblright\ the corresponding quotient maps, and the
quotient is exactly the remaining $\left\langle \left\langle Q,\psi
,R\right\rangle \right\rangle $-subschematic. The downward double exact
sequences are split exact. The resulting Nenashev relation is just%
\begin{equation}
\lbrack\left\langle \left\langle P,\varphi,Q\right\rangle \right\rangle
]-[J]+[\left\langle \left\langle Q,\psi,R\right\rangle \right\rangle
]=[0]-[0]+[0]\text{,}\label{lmiur6}%
\end{equation}
because the columns represent zero since Yin and Yang agree, Equation
\ref{l_C2_NenaAgreeIsZero}. We now set up maps from $J$ to a new schematic:
The idea is that we map the schematic to itself, just that on the objects
$T_{Q},\prod T_{Q},Q,\bigoplus Q$, which each appear twice in the shaded
areas, we use the identity map on the Yin side, but swap both objects on the
Yang side. On $P_{\mathbb{R}}$ and $Q_{\mathbb{R}}$ we use the identity both
for Yin and Yang. This changes the wiring, so actually this is not a map to
the same schematic, but to $J^{\prime}$, as depicted in Figure \ref{lmiur5}
above on the right. As before in the proof of Lemma \ref{sv1}, in this figure
we have left all objects where they were and just adjusted the arrows. The
solid arrows, i.e. the Yin side, of $J$ and $J^{\prime}$ agree, but we see
that on the Yang side all arrows have moved to their respective mirror
partner. The maps which we have just described, set up a $(3\times3)$-diagram%
\begin{equation}%
\begin{array}
[c]{c}%
J\\
\downdownarrows\\
J^{\prime}\\
\downdownarrows\\
0\text{,}%
\end{array}
\label{lmiur4}%
\end{equation}
and the quotient is zero since both the identity maps on the Yin side, as well
as the swapping maps on the Yang side are isomorphisms. The resulting Nenashev
relation is $[J]-[J^{\prime}]+[0]=[0]-[0]+[0]$, since again the columns are
either the identity or swapping maps, so they are zero by Lemma
\ref{Lemma_SwapIsZeroInLCAA}. Next, within the shaded area of Figure
\ref{lmiur5} on the right, we find the doubled versions of the the exact
sequences%
\[
\prod T_{Q}\overset{s}{\hookrightarrow}\prod T_{Q}\twoheadrightarrow
T_{Q}\qquad\text{and}\qquad\prod T_{Q}\overset{1}{\hookrightarrow}\prod
T_{Q}\twoheadrightarrow0\text{,}%
\]%
\[
Q\hookrightarrow\bigoplus Q\overset{s}{\twoheadrightarrow}\bigoplus
Q\qquad\text{and}\qquad0\hookrightarrow\bigoplus Q\overset{1}%
{\twoheadrightarrow}\bigoplus Q\text{,}%
\]
and note that both Yin and Yang morphisms agree. We can map them as a
sub-double exact sequence to $J^{\prime}$, giving a $(3\times3)$-diagram, but
since Yin and Yang agree, they all contribute the zero class. Thus, the class
of $[J^{\prime}]$ agrees with its quotient by these sequences, which yields
the schematic depicted below on the left:%
\[%
{\includegraphics[
height=2.5598in,
width=5.2797in
]%
{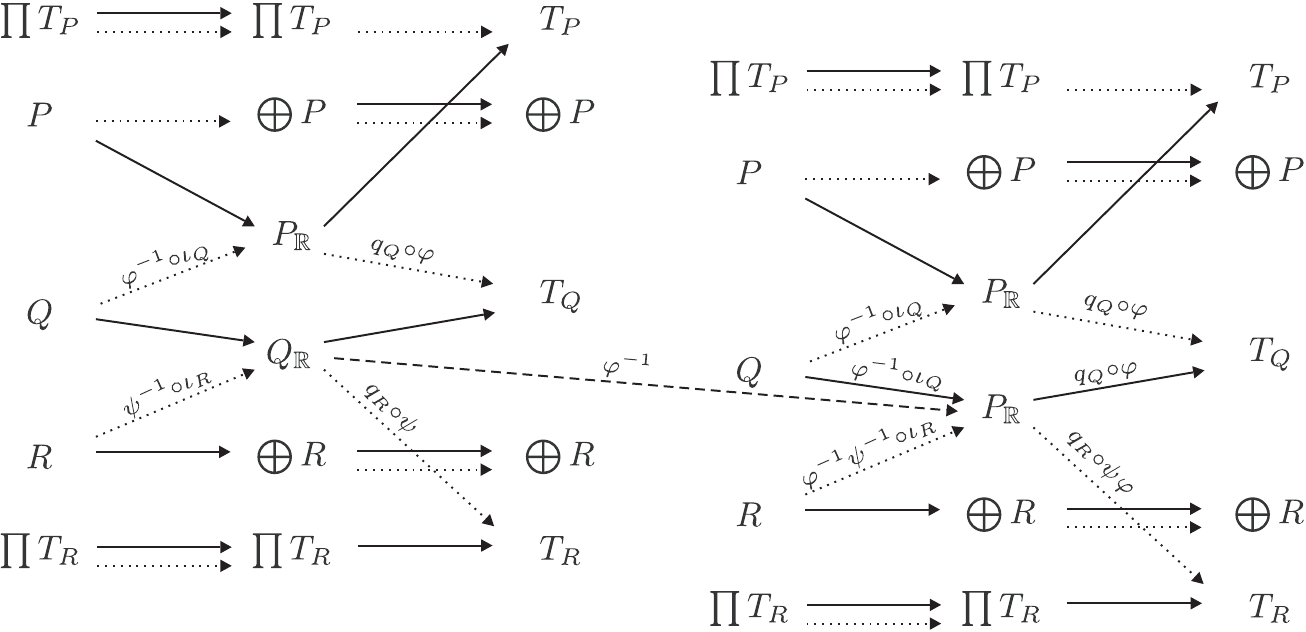}%
}
\]
We now map the schematic on the left to the one depicted above on the right.
To this end, we use identity maps on all objects except for $Q_{\mathbb{R}}$,
where we instead use $\varphi^{-1}$ both for Yin and Yang (as alluded to by
the dashed arrow in the figure). As the reader can see, we have adjusted the
wiring on the right hand side to ensure that this defines a commuting diagram
(obviously it suffices to change the four in- resp. outgoing maps around
$Q_{\mathbb{R}}$ resp. $P_{\mathbb{R}}$ to accomodate for these changes).
Write $J^{\prime\prime}$ for the schematic on the right. The morphism just
described gives rise to another $(3\times3)$-diagram as in Equation
\ref{lmiur4}, this time showing $[J^{\prime}]=[J^{\prime\prime}]$. There is no
contribution from the columns since we only used the identity map, or
$(\varphi^{-1},\varphi^{-1})$ for both Yin and Yang, which is zero by Equation
\ref{l_C2_NenaAgreeIsZero}. Finally, we set up a map from $J^{\prime\prime}$
(above on the right) to $J^{\prime\prime\prime}$, defined as%
\[%
{\includegraphics[
height=2.3281in,
width=2.4267in
]%
{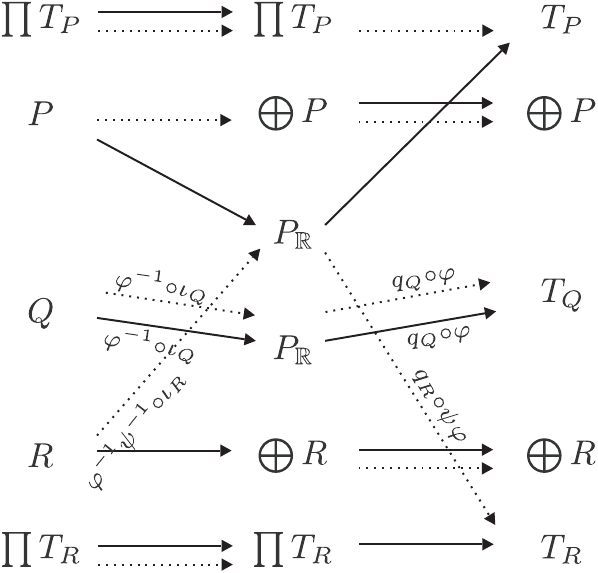}%
}
\]
This map is the identity on all objects on the Yin side. On the Yang side, it
is the identity on all objects, except on the pair $P_{\mathbb{R}}\oplus
P_{\mathbb{R}}$, on which it is the swapping map once more. The morphism again
gives rise to another $(3\times3)$-diagram as in Equation \ref{lmiur4},
showing $[J^{\prime\prime}]=[J^{\prime\prime\prime}]$ since there is no column
contribution (the columns are just the identity for both Yin and Yang, or the
swapping map, which is zero by Lemma \ref{Lemma_SwapIsZeroInLCAA} again). Just
as we had killed the pieces in the shaded areas in Figure \ref{lmiur5}, we can
now get rid off the double exact sequence $Q\hookrightarrow P_{\mathbb{R}%
}\twoheadrightarrow T_{R}$ in the middle, where the Yin and Yang side agree.
The resulting schematic is exactly $\left\langle \left\langle P,\psi
\varphi,R\right\rangle \right\rangle $. Thus, using that we had shown that
$[J]=[J^{\prime}]=[J^{\prime\prime}]=[J^{\prime\prime\prime}]$, Equation
\ref{lmiur6} now shows%
\[
\lbrack\left\langle \left\langle P,\varphi,Q\right\rangle \right\rangle
]-[\left\langle \left\langle P,\psi\varphi,R\right\rangle \right\rangle
]+[\left\langle \left\langle Q,\psi,R\right\rangle \right\rangle ]=0\text{,}%
\]
which proves that\ Relation B is respected under our map $\vartheta$.
\end{proof}

The combination of the previous lemmata settles the proof of Theorem
\ref{thm_VarThetaWellDefined}.

\section{Proof of isomorphy}

\subsection{Recollections on the Gillet--Grayson model\label{subsect_GG}}

Let $\mathsf{C}$ be a pointed exact category, i.e. an exact category with a
fixed choice of a zero object which we shall henceforth denote by
\textquotedblleft$0$\textquotedblright. We briefly summarize the
Gillet--Grayson model. It originates from the articles \cite{MR909784,
MR2007234}. Define a simplicial set $G_{\bullet}\mathsf{C}$ whose
$n$-simplices are given by a pair of commutative diagrams%
\[%
\xymatrix@!=0.157in{
&                         &                         &                        & P_{n/(n-1)}
\\
&                         &                         & \cdots\ar@{^{(}%
.>}[r] & \vdots\ar@{.>>}[u] \\
&                         & P_{2/1} \ar@{^{(}.>}[r] & \cdots\ar@{^{(}%
.>}[r] & P_{n/1} \ar@{.>>}[u] \\
& P_{1/0} \ar@{^{(}.>}[r] & P_{2/0} \ar@{^{(}.>}[r] \ar@{.>>}[u] & \cdots
\ar@{^{(}.>}[r] & P_{n/0} \ar@{.>>}[u] \\
P_0 \ar@{^{(}.>}[r] & P_1 \ar@{^{(}.>}[r] \ar@{.>>}[u] & P_2 \ar@{^{(}%
.>}[r] \ar@{.>>}[u] & \cdots\ar@{^{(}.>}[r] & P_n \ar@{.>>}[u]
}%
\qquad%
\xymatrix@!=0.157in{
&                         &                         &                        & P_{n/(n-1)}
\\
&                         &                         & \cdots\ar@{^{(}%
->}[r] & \vdots\ar@{->>}[u] \\
&                         & P_{2/1} \ar@{^{(}->}[r] & \cdots\ar@{^{(}%
->}[r] & P_{n/1} \ar@{->>}[u] \\
& P_{1/0} \ar@{^{(}->}[r] & P_{2/0} \ar@{^{(}->}[r] \ar@{->>}[u] & \cdots
\ar@{^{(}->}[r] & P_{n/0} \ar@{->>}[u] \\
P^{\prime}_0 \ar@{^{(}->}[r] & P^{\prime}_1 \ar@{^{(}->}[r] \ar@
{->>}[u] & P^{\prime}_2 \ar@{^{(}->}[r] \ar@{->>}[u] & \cdots\ar@{^{(}%
->}[r] & P^{\prime}_n \ar@{->>}[u]
}%
\text{,}%
\]
such that (1) the diagrams agree strictly above the bottom row, (2) all
sequences $P_{i}\hookrightarrow P_{j}\twoheadrightarrow P_{j/i}$ are exact,
(2') all sequences $P_{i}^{\prime}\hookrightarrow P_{j}^{\prime}%
\twoheadrightarrow P_{j/i}^{\prime}$ are exact, (3) all sequences
$P_{i/j}\hookrightarrow P_{m/j}\twoheadrightarrow P_{m/i}$ are exact.

The face and degeneracy maps come from deleting the $i$-th row and column, or
by duplicating them. For details we refer to the references. Thus, the
$0$-simplices are pairs $(P,P^{\prime})$ of objects. The $1$-simplices are
pairs of exact sequences%
\begin{equation}%
\xymatrix{
P_0 \ar@{^{(}.>}[r] & P_1 \ar@{.>>}[r] & P_{1/0} & \qquad& P^{\prime}%
_0 \ar@{^{(}->}[r] & P^{\prime}_1 \ar@{->>}[r] & P_{1/0}
}
\label{lcimez22}%
\end{equation}
with the same cokernel. The main result of Gillet and Grayson is the
equivalence%
\[
K(\mathsf{C})\cong\left\vert G_{\bullet}\mathsf{C}\right\vert \text{,}%
\]
or more specifically:\ The space $\left\vert G_{\bullet}\mathsf{C}\right\vert
$ is an infinite loop space and as such equivalent to a connective spectrum.
This spectrum is canonically equivalent to the $K$-theory spectrum of
$\mathsf{C}$. As explained in Remark \ref{rmk_CompatWithLiterature} the
$0$-simplex $(P,P^{\prime})$ lies in the connected component $[P^{\prime
}]-[P]\in\pi_{0}K(\mathsf{C})$.

\begin{definition}
[Nenashev]\label{def_e_loop}To any double exact sequence $\kappa$, Nenashev
attaches a loop $e(\kappa)\in\pi_{1}K(\mathsf{C})$. Concretely,%
\[
\kappa=\left[
\xymatrix{
A \ar@<1ex>@{^{(}->}[r]^{p} \ar@<-1ex>@{^{(}.>}[r]_{q} & B \ar@<1ex>@{->>}%
[r]^{r} \ar@<-1ex>@{.>>}[r]_{s} & C
}%
\right]
\]
is mapped to a path made from three $1$-simplices in the Gillet--Grayson
model, namely

\begin{enumerate}
\item we go from $(0,0)$ to $(A,A)$ by the edge
\[%
\xymatrix{
0 \ar@{^{(}.>}[r] & A \ar@{.>>}[r]^{1} & A & \qquad& 0 \ar@{^{(}%
->}[r] & A \ar@{->>}[r]^{1} & A
}%
\text{,}%
\]

\item then we go from $(A,A)$ to $(B,B)$ by the edge%
\[%
\xymatrix{
A \ar@{^{(}.>}[r]^{q} & B \ar@{.>>}[r]^{s} & C & \qquad& A \ar@{^{(}->}[r]^{p}
& B \ar@{->>}[r]^{r} & C
}%
\]

\item and then we return from $(B,B)$ to $(0,0)$ by running backwards along%
\[%
\xymatrix{
0 \ar@{^{(}.>}[r] & B \ar@{.>>}[r]^{1} & B & \qquad& 0 \ar@{^{(}%
->}[r] & B \ar@{->>}[r]^{1} & B
}%
\text{.}%
\]

\end{enumerate}

If $A\in\mathsf{C}$ is an object, $e(A)$ denotes the edge%
\[%
\xymatrix{
0 \ar@{^{(}.>}[r] & A \ar@{.>>}[r]^{1} & A & \qquad& 0 \ar@{^{(}%
->}[r] & A \ar@{->>}[r]^{1} & A
}%
\]
from $(0,0)$ to $(A,A)$.
\end{definition}

The path $e(\kappa)$ is visibly a closed loop. All vertices lie in the
connected component of zero in $\left\vert G_{\bullet}\mathsf{C}\right\vert $.
This construction agrees with the $e(\kappa)$ in \cite[Equation 2.2 $\frac
{1}{2}$]{MR1637539}, except for the swapped roles of left and right, in line
with Remark \ref{rmk_CompatWithLiterature}. In summary: Whenever we use a
double exact sequence in the Nenashev presentation, it corresponds
homotopically to the closed loop just described.

\subsection{Proof}

\begin{theorem}
\label{thm_IdentifySeq}Suppose $\mathfrak{A}$ is a regular order in a
semisimple finite-dimensional $\mathbb{Q}$-algebra $A$. Then there is a
commutative diagram%
\begin{equation}%
\xymatrix{
\cdots\ar[r] & K_{1}(\mathfrak{A},{\mathbb{R} }) \ar[r] \ar[d] & K_{1}%
(\mathfrak{A}) \ar@{=}[d] \ar[r] & K_{1}(A_{{\mathbb{R} }}) \ar@{=}%
[d] \ar[r]^{\delta} \ar@{}[dr]|{X} & K_{0}(\mathfrak{A},{\mathbb{R} }%
) \ar[d]^{\vartheta} \ar[r] \ar@{}[dr]|{Y} & K_{0}(\mathfrak{A}) \ar@
{=}[d] \ar[r] & \cdots\\
\cdots\ar[r] & K_{2}(\mathsf{LCA}_{\mathfrak{A}}) \ar[r] & K_{1}(\mathfrak
{A}) \ar[r] & K_{1}(A_{{\mathbb{R} }}) \ar[r] & K_{1}(\mathsf{LCA}%
_{\mathfrak{A}}) \ar[r]_{\partial} & K_{0}(\mathfrak{A}) \ar[r] & \cdots\\
}
\label{lcuup1}%
\end{equation}
and the map $\vartheta$ of Theorem \ref{thm_VarThetaWellDefined} is an isomorphism.
\end{theorem}

We split the proof into several lemmata.

\begin{lemma}
\label{sw1}The square denoted by $X$ in Equation \ref{lcuup1} commutes.
\end{lemma}

\begin{proof}
Let an element in $K_{1}(A_{\mathbb{R}})$ be given, say represented by some
isomorphism $\varphi\in\operatorname{GL}_{n}(A_{\mathbb{R}})$ for $n$
sufficiently large. We follow the downward and then right arrow first: The
downward arrow is the identity. The bottom map is induced from the exact
functor $\operatorname*{PMod}(A_{\mathbb{R}})\rightarrow\mathsf{LCA}%
_{\mathfrak{A}}$ which sends a finitely generated projective right
$A_{\mathbb{R}}$-module to itself, regarded as a locally compact module with
the real topology (since $A_{\mathbb{R}}$ is semisimple, finitely generated
projective just means free of finite rank here). This class corresponds to an
automorphism and by Remark \ref{rmk_CompatWithLiterature} it has the Nenashev
representative%
\begin{equation}%
\xymatrix{
0 \ar@<1ex>@{^{(}->}[r]^-{0} \ar@<-1ex>@{^{(}.>}[r]_-{0} & {A_{\mathbb{R} }^n}
\ar@<1ex>@{->>}[r]^{\varphi} \ar@<-1ex>@{.>>}[r]_{1} & {{A_{\mathbb{R} }^n}%
}\text{.}
}%
\label{lmiur22}%
\end{equation}
Next, follow the square the other way: Following the top right arrow, we get%
\[
\delta(\varphi)=[\mathfrak{A}^{n},\varphi,\mathfrak{A}^{n}]\in K_{0}%
(\mathfrak{A},\mathbb{R})\text{,}%
\]
using the explicit presentation, see \cite[p. 215]{MR0245634}. The map
$\vartheta$ sends this to the class of $\left\langle \left\langle
\mathfrak{A}^{n},\varphi,\mathfrak{A}^{n}\right\rangle \right\rangle $, which
in turn unravels as the double exact sequence underlying the schematic
depicted below on the left:%
\[%
{\includegraphics[
height=1.8593in,
width=5.463in
]%
{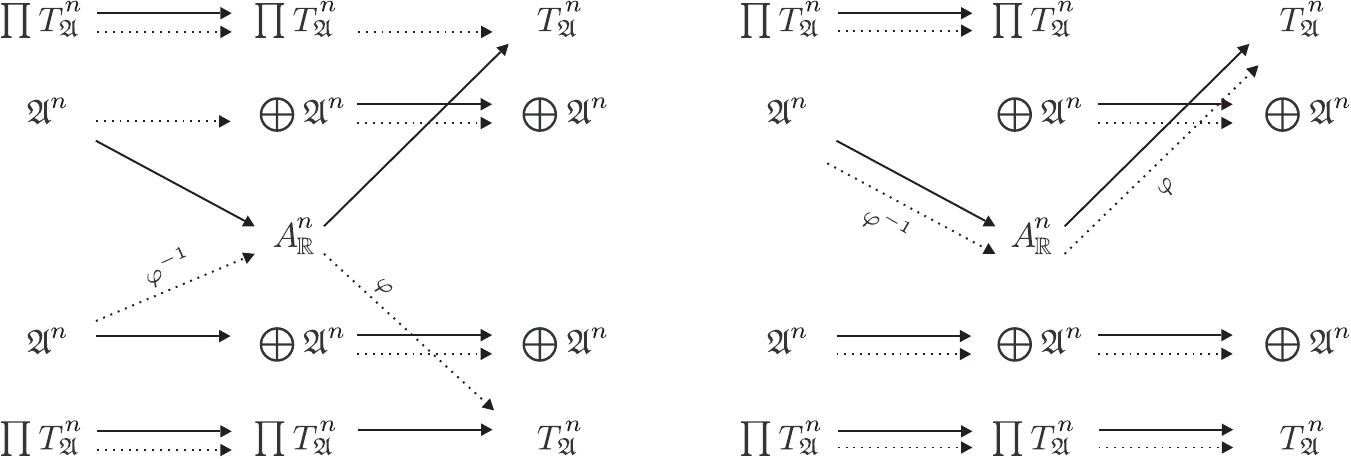}%
}
\]
We now show that its class agrees with the class of the schematic depicted
above on the right. The proof for this follows exactly the same pattern as the
proof of Lemma \ref{sv1} and we leave the details to the reader. However,
unlike in the cited proof, this time the middle double exact sequence (the one
which is bent in the figure above on the right) can be non-trivial. We use it
as the first row in the following $(3\times3)$-diagram:%
\[%
\xymatrix@W=0.3in@H=0.3in{
{\mathfrak{A}^n } \ar@<1ex>@{^{(}->}[r]^{ {\iota_{{\mathfrak{A}^n }}}}
\ar@<1ex>@{^{(}.>}[d]^{1}
\ar@<-1ex>@{^{(}->}[d]_{1} \ar@<-1ex>@{^{(}.>}[r]_{\varphi^{-1} }
& {A_{\mathbb{R}}^n } \ar@<1ex>@{->>}[r]^{ {q_{{\mathfrak{A}^n }}} }
\ar@<-1ex>@{.>>}[r]_{\varphi}
\ar@<-1ex>@{^{(}->}[d]_{1} \ar@<1ex>@{^{(}.>}[d]^{\varphi} & {T_{\mathfrak{A}
}^n } \ar@<1ex>@{^{(}.>}[d]^{1} \ar@<-1ex>@{^{(}->}[d]_{1} \\
{\mathfrak{A}^n } \ar@<1ex>@{^{(}->}[r]^{ {\iota_{{\mathfrak{A}^n }}} }
\ar@<-1ex>@{^{(}.>}[r]_{ {\iota_{{\mathfrak{A}^n }}} }
\ar@<1ex>@{.>>}[d] \ar@<-1ex>@{->>}[d] & {A_{\mathbb{R}}^n } \ar@
<1ex>@{->>}[r]^{ {q_{{\mathfrak{A}^n }}} } \ar@<-1ex>@{.>>}%
[r]_{ {q_{{\mathfrak{A}^n }}}} \ar@<1ex>@{.>>}[d] \ar@<-1ex>@{->>}%
[d] & {T_{\mathfrak{A} }^n } \ar@<1ex>@{.>>}[d] \ar@<-1ex>@{->>}[d] \\
0 \ar@<1ex>@{^{(}->}[r] \ar@<-1ex>@{^{(}.>}[r] & 0 \ar@<1ex>@{->>}%
[r] \ar@<-1ex>@{.>>}[r] & 0 \\
}%
\]
Using that the top row of this diagram represents the class of $\left\langle
\left\langle \mathfrak{A}^{n},\varphi,\mathfrak{A}^{n}\right\rangle
\right\rangle $ as we had just shown, the resulting Nenashev relation, i.e.
the relation of Equation \ref{l_C_Nenashev}, of this diagram reads%
\[
\lbrack\left\langle \left\langle \mathfrak{A}^{n},\varphi,\mathfrak{A}%
^{n}\right\rangle \right\rangle ]-[0]+[0]=[0]-\left[
\xymatrix{
{{A_{\mathbb{R} }^n}} \ar@<1ex>@{^{(}->}[r]^-{1} \ar@<-1ex>@{^{(}%
.>}[r]_-{\varphi} & {A_{\mathbb{R} }^n} \ar@<1ex>@{->>}[r]^-{1} \ar@
<-1ex>@{.>>}[r]_-{1} & 0\text{.}
}%
\right]  +[0]\text{,}%
\]
but this agrees with Equation \ref{lmiur22} by Lemma \ref{lemma_leftrightswap}.
\end{proof}

We recall the following definition since it plays a fairly important r\^{o}le
in the next proof.

\begin{definition}
\label{def_cg}Suppose $G\in\mathsf{LCA}$ is an LCA group. A subset $U\subseteq
G$ is called \emph{symmetric} if $g\in U$ implies $-g\in U$. The group $G$ is
called \emph{compactly generated} if there exists a symmetric compact subset
$C\subseteq G$ such that $G=\bigcup_{n\geq0}C^{n}$. We write $\mathsf{LCA}%
_{\mathfrak{A},cg}$ for the fully exact subcategory of topological right
$\mathfrak{A}$-modules whose underlying LCA group is compactly generated.
\end{definition}

\begin{lemma}
\label{sw2}The square denoted by $Y$ in Equation \ref{lcuup1} commutes.
\end{lemma}

\begin{proof}
Let $[P,\varphi,Q]\in K_{0}(\mathfrak{A},\mathbb{R})$ be an arbitrary element.
Then under the top right arrow it is sent to%
\begin{equation}
\lbrack P]-[Q]\in K_{0}(\mathfrak{A})\text{,}\label{lmiur23}%
\end{equation}
see \cite[p. 215]{MR0245634} or \cite[Chapter II, Definition 2.10]{MR3076731}
for this map. Thus, it remains to see that we also get this if we follow the
square the other way:\ Firstly, the map $\vartheta$ sends $[P,\varphi,Q]$ to
the class of the double exact sequence underlying $\left\langle \left\langle
P,\varphi,Q\right\rangle \right\rangle $, i.e. the schematic%
\begin{equation}%
{\includegraphics[
height=1.855in,
width=2.4154in
]%
{gfx_1-eps-converted-to.pdf}%
}
\label{lmiur20_1}%
\end{equation}
Thus, we need to compute what the map $\partial$ in Diagram \ref{lcuup1} does
with this element. We begin by describing how the boundary map%
\[
\partial:K_{1}(\mathsf{LCA}_{\mathfrak{A}})\longrightarrow K_{0}(\mathfrak{A})
\]
arises in \cite{etnclca}. For this, we have to trace through our constructions
and go back to how this map was defined. This leads us to the proof of
\cite[Proposition 11.1]{etnclca}. In this proof, we first set up the diagram
\[%
\xymatrix{
{\mathsf{Mod}_{{\mathfrak{A}},fg}} \ar[r] \ar[d]_{g} & {\mathsf{Mod}%
_{\mathfrak{A}}} \ar[r] \ar[d] & {{\mathsf{Mod}_{\mathfrak{A}}}}%
/{{\mathsf{Mod}_{{\mathfrak{A}},fg}}} \ar[d]^{\Phi} \\
\mathsf{LCA}_{\mathfrak{A},cg} \ar[r] & \mathsf{LCA}_{\mathfrak{A}}
\ar[r] & {\mathsf{LCA}_{\mathfrak{A}}}/{\mathsf{LCA}_{\mathfrak{A},cg}}
\text{,}
}%
\]
where $\mathsf{Mod}_{\mathfrak{A},fg}$ denotes the category of finitely
generated right $\mathfrak{A}$-modules, $\mathsf{Mod}_{\mathfrak{A}}$ the
category of all right $\mathfrak{A}$-modules, $\mathsf{LCA}_{\mathfrak{A}%
,cg}=\mathsf{LCA}_{\mathfrak{A}}\cap\mathsf{LCA}_{cg}$ the category of
topological right $\mathfrak{A}$-modules whose underlying LCA group is
compactly generated. Both rows are exact sequences of exact categories. The
map $\Phi$ stems from the exact functor%
\[
\mathsf{Mod}_{\mathfrak{A}}/\mathsf{Mod}_{\mathfrak{A},fg}\longrightarrow
\mathsf{LCA}_{\mathfrak{A}}/\mathsf{LCA}_{\mathfrak{A},cg}\text{,}%
\]
which is induced from sending a right $\mathfrak{A}$-module to itself,
equipped with the discrete topology. It is shown loc. cit. that $\Phi$ is in
fact an exact equivalence of the given quotient exact categories. In
particular, it induces an equivalence of the level of $K$-theory. Applying
non-connective $K$-theory, we get two fiber sequences and then the argument
ibid. shows that the left square is bi-Cartesian in spectra. The long exact
sequences in Diagram \ref{lcuup1} now arise as the long exact sequences of
(stable) homotopy groups, using that the bi-Cartesian diagram has a
contractible node. From this, one unravels by diagram chase that the
differential%
\[
\partial:\pi_{1}K(\mathsf{LCA}_{\mathfrak{A}})\longrightarrow\pi
_{0}K\mathbb{(}\mathsf{Mod}_{\mathfrak{A},fg})
\]
agrees with the composition $\partial^{\ast}\circ\Phi^{-1}\circ q$ in the
following diagram%
\[%
\xymatrix{
\cdots\ar[r] & {{\pi}_1}K({\mathsf{Mod}_{\mathfrak{A}}}) \ar[r] \ar
[d] & {{\pi}_1}K({{\mathsf{Mod}_{\mathfrak{A}}}}/{{\mathsf{Mod}_{{\mathfrak
{A}},fg}}}) \ar@{=}[d]^{\Phi} \ar[r]^-{{\partial}^{\ast}} & {{\pi}%
_0}K({\mathsf{Mod}_{{\mathfrak{A}},fg}}) \ar[d] \\
{{\pi}_1}K(\mathsf{LCA}_{\mathfrak{A},cg}) \ar[r] & {{\pi}_1}K(\mathsf
{LCA}_{\mathfrak{A}}) \ar[r]_-{q} & {{\pi}_1}K({\mathsf{LCA}_{\mathfrak{A}}%
}/{\mathsf{LCA}_{\mathfrak{A},cg}}) \ar[r] & {{\pi}_0}K(\mathsf{LCA}%
_{\mathfrak{A},cg}),
}%
\]
where $\partial^{\ast}$ is the boundary map of the long exact sequence of
homotopy groups coming from the localization sequence in the top row, i.e.%
\[
\mathsf{Mod}_{\mathfrak{A},fg}\longrightarrow\mathsf{Mod}_{\mathfrak{A}%
}\longrightarrow\mathsf{Mod}_{\mathfrak{A}}/\mathsf{Mod}_{\mathfrak{A}%
,fg}\text{.}%
\]
This localization sequence ibid. was produced from\ Schlichting's localization
theorem, see \cite[Theorem 4.1]{obloc} for the formulation we use, or
\cite{MR2079996} for the proof. Several further remarks: The proof of
\cite[Theorem 11.3]{etnclca} shows that the non-connective $K$-theory of all
the involved categories agrees with their connective $K$-theory, i.e. their
usual Quillen algebraic $K$-theory. Thus, from now on view $K(-)$ as a
connective spectrum. We now unravel the maps in $\partial^{\ast}\circ\Phi
^{-1}\circ q$:\newline(Step 1) The map $q$ is induced from the exact quotient
functor%
\[
\mathsf{LCA}_{\mathfrak{A}}\longrightarrow\mathsf{LCA}_{\mathfrak{A}%
}/\mathsf{LCA}_{\mathfrak{A},cg}\text{.}%
\]
Thus, $q$ sends the schematic in Figure \ref{lmiur20_1} to itself, but
regarded in $\mathsf{LCA}_{\mathfrak{A}}/\mathsf{LCA}_{\mathfrak{A},cg}$. We
observe that $P$ is finitely generated projective and thus has underlying
LCA\ group isomorphic to $\mathbb{Z}^{n}$ for some $n\in\mathbb{Z}_{\geq0}$.
Since $T_{P}$ is a torus, it is compact, and by Tychonoff's Theorem, $\prod
T_{P}$ is also compact. Both arguments also work verbatim for $Q$ and $\prod
T_{Q}$. Further, the underlying LCA group of $P_{\mathbb{R}}$ is a
finite-dimensional real vector space. Thus, by the classification of compactly
generated LCA\ groups, all these objects are compactly generated, see
\cite[Theorem 2.5]{MR0215016}. As a result, for each such object we obtain an
isomorphism to the zero object in the category $\mathsf{LCA}_{\mathfrak{A}%
}/\mathsf{LCA}_{\mathfrak{A},cg}$. Thus, the class of $q(\left\langle
\left\langle P,\varphi,Q\right\rangle \right\rangle )$ is the double exact
sequence $l$, defined as%
\begin{equation}%
\begin{tabular}
[c]{ccccc}%
$0$ & $\longrightarrow$ & $\bigoplus P$ & $\overset{1}{\longrightarrow}$ &
$\bigoplus P$\\
$0$ & $\longrightarrow$ & $\bigoplus Q$ & $\overset{s}{\longrightarrow}$ &
$\bigoplus Q$\\\hline
$0$ & $\longrightarrow$ & $\bigoplus P$ & $\overset{s}{\longrightarrow}$ &
$\bigoplus P$\\
$0$ & $\longrightarrow$ & $\bigoplus Q$ & $\overset{1}{\longrightarrow}$ &
$\bigoplus Q$%
\end{tabular}
\ \ \ \ \ \label{lmiur20_2}%
\end{equation}
in the notation of Equation \ref{l_Schematic} (and the wiring is the obvious
one, that is: each object above the line is wired to its copy below the line;
this is the wiring which remains from Figure \ref{lmiur20_1}). We recall that
$s$ is the shift map of Equation \ref{l_Map_Shift}. Since the wiring is so
simple here, we stick to this graphical presentation from now on. It may
appear like a mistake that the shift maps $s$ are both depicted with kernel
\textquotedblleft$0$\textquotedblright, but this is accurate since while in
the category $\mathsf{LCA}_{\mathfrak{A}}$ the quotient would be $P$ resp.
$Q$, in the quotient category $\mathsf{LCA}_{\mathfrak{A}}/\mathsf{LCA}%
_{\mathfrak{A},cg}$ these objects are isomorphic to zero. Next, we need to
apply the map $\Phi^{-1}$. Being the inverse map to a functor, this is usually
a tough operation, but it is easy in our situation: Note that all objects in
Figure \ref{lmiur20_2} are discrete right $\mathfrak{A}$-modules, so we can
right away read this as the schematic for a double exact sequence in
$\mathsf{Mod}_{\mathfrak{A}}/\mathsf{Mod}_{\mathfrak{A},fg}$, which $\Phi$
obviously sends to itself. Finally, we need to apply $\partial^{\ast}$. This
is more delicate: Recall that%
\[
K\mathbb{(}\mathsf{Mod}_{\mathfrak{A},fg})\longrightarrow K(\mathsf{Mod}%
_{\mathfrak{A}})\longrightarrow K\mathbb{(}\mathsf{Mod}_{\mathfrak{A}%
}/\mathsf{Mod}_{\mathfrak{A},fg})
\]
is a fiber sequence in spectra by Schlichting's localization theorem. As all
spectra here are connective, we may read this as a fiber sequence of infinite
loop spaces as well. We take simplicial sets $\mathsf{sSet}$ as our $\infty
$-category of `spaces'. Next, use Gillet--Grayson's model to have a simplicial
set describing $K$-theory; we had reviewed what we need in \S \ref{subsect_GG}%
. Now, the boundary map%
\begin{equation}
\partial^{\ast}:\pi_{1}K(\mathsf{Mod}_{\mathfrak{A}}/\mathsf{Mod}%
_{\mathfrak{A},fg})\longrightarrow\pi_{0}K\mathbb{(}\mathsf{Mod}%
_{\mathfrak{A},fg})\label{lmiur21_2}%
\end{equation}
has the following topological meaning: We take a loop $\ell$ in the space
$K(\mathsf{Mod}_{\mathfrak{A}}/\mathsf{Mod}_{\mathfrak{A},fg})$ and lift it
along the fibration of spaces%
\[
K(\mathsf{Mod}_{\mathfrak{A}})\longrightarrow K(\mathsf{Mod}_{\mathfrak{A}%
}/\mathsf{Mod}_{\mathfrak{A},fg})\text{.}%
\]
This is possible, but only as a path which need not be closed, so we obtain a
path from the distinguished pointing of $K(\mathsf{Mod}_{\mathfrak{A}})$,
which is the point defined by $(0,0)$ for the zero object in the
Gillet--Grayson model, to some other point, which in concrete terms is a pair
$(A,B)$. The boundary map $\partial^{\ast}$ then sends $\ell$ to the connected
component in which this endpoint $(A,B)$ lies. Since we use the
Gillet--Grayson model, we need to perform all these steps simplicially. In the
case at hand, the double exact sequence $\Phi^{-1}q(\left\langle \left\langle
P,\varphi,Q\right\rangle \right\rangle )$ is our input loop, and we had seen
that the description of Figure \ref{lmiur20_2} is valid. In the
Gillet--Grayson model, a double exact sequence corresponds to the path%
\begin{equation}%
{\includegraphics[
height=1.2341in,
width=3.0969in
]%
{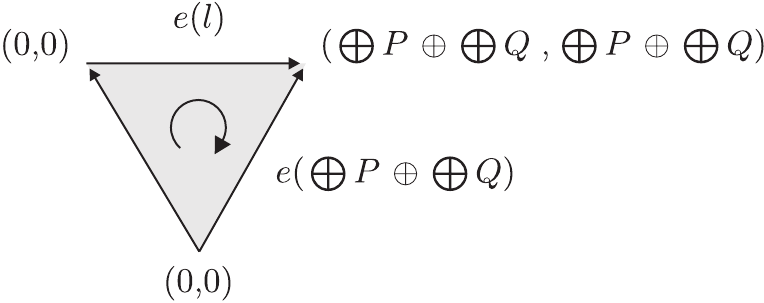}%
}
\label{lmiur21}%
\end{equation}
See Definition \ref{def_e_loop}. Recall that $l$ was the double exact sequence
of Figure \ref{lmiur20_2}. As explained above, we now need to lift this loop
to a path in the Gillet--Grayson space of $\mathsf{Mod}_{\mathfrak{A}}$. We
recall that an edge from the point $(P_{0},P_{0}^{\prime})$ to a point
$(P_{1},P_{1}^{\prime})$ in the Gillet--Grayson model corresponds to a pair of
exact sequences with equal cokernels. Concretely, we shall take%
\begin{equation}%
\xymatrix{
Q \ar@{^{(}.>}[r] & \bigoplus P\oplus\bigoplus Q \ar@{.>>}[r]^-{1 \oplus s}
& \bigoplus P\oplus\bigoplus Q & & P \ar@{^{(}->}[r] & \bigoplus
P\oplus\bigoplus Q \ar@{->>}[r]^-{s \oplus1}
& \bigoplus P\oplus\bigoplus Q
}%
\label{l_55}%
\end{equation}
which visibly have equal cokernels. This defines an edge (we stress that this
datum is different from a double exact sequence). We now replace the edge
$e(l)$ in Figure \ref{lmiur21} by this new edge. We get a non-closed path%
\begin{equation}%
{\includegraphics[
height=1.1243in,
width=3.1401in
]%
{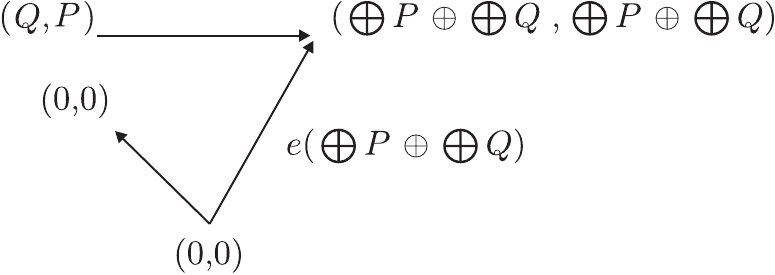}%
}
\label{lmiur32}%
\end{equation}
The key point is that this path also defines a path in $\pi_{1}K(\mathsf{Mod}%
_{\mathfrak{A}})$, much unlike Figure \ref{lmiur21} which would not since the
shift map $s$ is \textit{not} an isomorphism in the category $\mathsf{Mod}%
_{\mathfrak{A}}$. We claim that it is a lift along the fibration:\ Indeed,
under the map%
\[
\pi_{1}K(\mathsf{Mod}_{\mathfrak{A}})\longrightarrow\pi_{1}K(\mathsf{Mod}%
_{\mathfrak{A}}/\mathsf{Mod}_{\mathfrak{A},fg})
\]
the kernels $P$ and $Q$ which appear on the left in Equation \ref{l_55} again
become isomorphic to zero since they are finitely generated, and thus the
upper edge transforms to $e(l)$ again. As we had discussed above, the boundary
map $\partial^{\ast}$ of Equation \ref{lmiur21_2} now maps the lifted path to
the connected component of its endpoint. This endpoint is $(Q,P)$. However,
the identification%
\[
\pi_{0}K(\mathsf{Mod}_{\mathfrak{A},fg})\overset{\sim}{\longrightarrow}%
K_{0}(\mathsf{Mod}_{\mathfrak{A},fg})
\]
in terms of the Gillet--Grayson model is given by the map $(Q,P)\mapsto\lbrack
P]-[Q]$, see Remark \ref{rmk_CompatWithLiterature}. Hence, we have obtained
the same as in Equation \ref{lmiur23}, which is exactly what we had to show.
\end{proof}

\begin{proof}
[Proof of Theorem \ref{thm_IdentifySeq}]Using Lemma \ref{sw1} and Lemma
\ref{sw2}, we find that the two squares to the left as well as to the right of
the downward arrow $\vartheta$ in Equation \ref{lcuup1} commute. Since all
other downward arrows in these squares are isomorphisms, it follows that
$\vartheta$ is an isomorphism as well by the Five Lemma.
\end{proof}

\begin{conjecture}
There should be an elementary proof that $\vartheta$ is an isomorphism, solely
based on the generators and relations, and without truly touching algebraic
$K$-theory.
\end{conjecture}

More might be possible:

\begin{conjecture}
There should be a recipe to map a double exact sequence to a Bass--Swan
generator. It should be possible to check in an elementary (but probably
complicated) fashion that composing both maps either way yields the identity map.
\end{conjecture}

\begin{remark}
Using the work of Grayson \cite{MR2947948}, the higher $K$-groups
$K_{n}(\mathsf{LCA}_{\mathfrak{A}})$ all have explicit presentations
generalizing Nenashev's presentation for $K_{1}$. The compatibility between
the Nenashev and Grayson presentations is settled first in \cite{kaspwinges},
and even better in \cite{kaspkoeckwinges}. Furthermore, the long exact
sequence%
\[
\cdots\longrightarrow K_{n}(\mathfrak{A})\longrightarrow K_{n}(A_{\mathbb{R}%
})\longrightarrow K_{n-1}(\mathfrak{A},\mathbb{R})\longrightarrow
K_{n-1}(\mathfrak{A})\longrightarrow\cdots
\]
involving relative $K$-groups can also be generalized to higher $n>0$, and the
higher relative $K$-groups $K_{n-1}(\mathfrak{A},\mathbb{R})$ admit a similar
concrete model in the style of Grayson, see \cite[Corollary 2.3]%
{graysonrelative}. Perhaps one could extend the above proof and isolate a
concrete formula for the isomorphism%
\[
K_{n}(\mathfrak{A},\mathbb{R})\overset{\sim}{\longrightarrow}K_{n+1}%
(\mathsf{LCA}_{\mathfrak{A}})\text{,}%
\]
of Theorem \ref{thm_IdentifySeq} for all $n\geq0$. It would be based on the
corresponding relative versus absolute Grayson presentation using binary multi-complexes.
\end{remark}

\bibliographystyle{amsalpha}
\bibliography{ollinewbib}
\end{document}